\documentclass{article}

\usepackage{fixmath}

\usepackage{csquotes}
\usepackage[english]{babel}

\usepackage{amsmath}
\usepackage{amssymb,latexsym,stmaryrd}

\usepackage{enumitem}

\usepackage{tikz}
\usetikzlibrary{matrix,arrows,calc,fit,decorations.pathmorphing,fit,shapes.geometric,cd,external}
\usepackage{calc}
\usepackage{array}

\usepackage{amsthm} 

\usepackage[capitalise]{cleveref}

\usepackage[style=alphabetic, bibstyle=alphabetic, 
backend=bibtex8, bibencoding=ascii, 
url=false,doi=false,isbn=false]{biblatex}

\theoremstyle{plain} 

\newtheorem{theorem}{Theorem}[section] 
\newtheorem{proposition}[theorem]{Proposition}
\newtheorem{corollary} [theorem]{Corollary}

\newtheorem{lemma} [theorem]{Lemma}
\newtheorem{fact}[theorem]{Fact}

\newtheorem{conjecture}{Conjecture}
\newtheorem{problem}[conjecture]{Problem}

\theoremstyle{definition}
\newtheorem{definition}[theorem]{Definition}
\newtheorem{notation}[theorem]{Notation} 
 	
\newtheorem{definitions}[theorem]{Definitions}	

\theoremstyle{remark}
\newtheorem{remark}[theorem]{Remark}

\newtheorem{example}[theorem]{Example}
\newtheorem{examples}[theorem]{Examples}

\newlist{enumThm}{enumerate}{2}
\setlist[enumThm,1]{label=\upshape(\roman*)}
\setlist[enumThm,2]{label=\upshape(\alph*)}

\newlist{enumWqo}{enumerate}{1}
\setlist[enumWqo]{labelindent=\parindent, leftmargin=*,label=\normalfont\upshape (W\arabic*)}

\newlist{enumDefn}{enumerate}{2}
\setlist[enumDefn,1]{label=(\arabic*)}
\setlist[enumDefn,2]{label=\arabic{enumDefni}.\alph*,leftmargin=3em}

\newlist{enumRems}{enumerate}{1}
\setlist[enumRems,1]{label=\itshape\roman*.}

\newlist{enumExs}{enumerate}{1}
\setlist[enumExs,1]{label=\upshape(\alph*)}

\newlist{itemProof}{itemize}{1}
\setlist[itemProof]{label=--}

\newlist{enumProof}{enumerate}{1}
\setlist[enumProof,1]{label=\upshape(\alph*)}

\newlist{enumProofParameter}{enumerate}{1}
\setlist[enumProofParameter,1]{label=\upshape(\alph*)$_{\xi}$}

\newlist{descThm}{description}{1}
\setlist[descThm]{font=\normalfont}

\newcommand{\SemiCol}[1]{#1:}

\newlist{descProofUp}{description}{1}
\setlist[descProofUp]{font={\normalfont\SemiCol}}

\newcommand{\itSemiCol}[1]{\textit{#1:}}

\newlist{descProof}{description}{1}
\setlist[descProof]{font={\normalfont\itSemiCol}}

\newcommand{\ts}[1]{\textsuperscript{#1}}

\makeatletter
\def\enumfix{%
\if@inlabel
 \noindent\par\nobreak\vskip-\parskip\vskip-\parskip\hrule\@height\z@
\fi}
\makeatletter
\if@cref@capitalise

\crefname{IntroTheorem}{Theorem}{Theorems}
\crefname{section}{Section}{Sections}
\Crefname{section}{Section}{Sections}
\crefname{paragraph}{Paragraph}{Paragraphs}
\Crefname{paragraph}{Paragraph}{Paragraphs}
\crefname{subsection}{Subsection}{Subsections}
\Crefname{subsection}{Subsection}{Subsections}
\crefname{equation}{}{}

\crefname{chapter}{Chapter}{Chapters}

\else

\crefname{chapter}{chapter}{chapters}
\crefname{section}{section}{sections}
\crefname{subsection}{subsection}{subsections}
\crefname{paragraph}{paragraph}{paragraphs}
\crefname{equation}{}{}

\fi

\makeatother

\Crefname{figure}{Figure}{Figure}

\crefformat{enumWqoi}{#2#1#3}
 \Crefformat{enumWqoi}{#2#1#3}
\crefname{enumWqoi}{}{}

\crefformat{enumDefni}{#2#1#3}
 \Crefformat{enumDefni}{#2#1#3}
\crefname{enumDefni}{}{}
\crefrangeformat{enumDefni}{#3#1#4--#5#2#6}

\crefformat{enumDefnii}{#2#1#3}
 \Crefformat{enumDefnii}{#2#1#3}
\crefname{enumDefnii}{}{}
 
 \crefformat{enumThmi}{#2#1#3}
 \Crefformat{enumThmi}{#2#1#3}
\crefname{enumThmi}{}{}

 \crefformat{enumThmii}{#2#1#3}
 \Crefformat{enumThmii}{#2#1#3}
\crefname{enumThmii}{}{}

 \crefformat{enumRemsi}{#2#1#3}
 \Crefformat{enumRemsi}{#2#1#3}
\crefname{enumRemsi}{}{}

 \crefformat{enumProofi}{#2#1#3}
 \Crefformat{enumProofi}{#2#1#3}
\crefname{enumThmi}{}{}
 
\newcommand{\front}{front\ts{e}}

\newcommand{\wqo}{\textsc{wqo}}
\newcommand{\bqo}{\textsc{bqo}}

\newcommand{\contMor}{\leq_{\text{ch}}}
\newcommand{\NcontMor}{\not\leq_{\text{ch}}}
\newcommand{\contMorEq}{\equiv_{\text{ch}}}

\newcommand{\CtblSubs}{\mathcal{P}_{<\aleph_{1}}}

\newcommand{\PowerQ}{V^{*}}
\newcommand{\GPowerQ}{G_{V^{*}}}
\newcommand{\HerCtblQ}{H^{*}_{\omega_{1}}}

\DeclareMathOperator{\supp}{supp_{Q}}

\newcommand{\incomp}{\mathbin{|}}

\newcommand{\Rado}{\mathfrak{R}}

\newcommand{\conc}{\mathrel{{}^{\smallfrown}}} 
\newcommand{\restr}[1]{\mathord{\upharpoonright_{#1}}} 

\DeclareMathOperator{\id}{\mathrm{id}} 

\newcommand{\bai}{\omega^{\omega}}

\DeclareMathOperator*{\Seq}{\mathrm{seq}}

\newcommand{\SSI}{\quad\longleftrightarrow\quad} 

\DeclareMathOperator{\base}{\mathord{\bigcup}}

\DeclareMathOperator{\rk}{rk}

\DeclareMathOperator{\upcl}{\uparrow}
\DeclareMathOperator{\dcl}{\downarrow}

\DeclareMathOperator{\Down}{\mathcal{D}} 
\DeclareMathOperator{\DownFb}{\Down_{\text{fb}}} 
\DeclareMathOperator{\Final}{\mathcal{U}} 

\newcommand{\segm} {\sqsubseteq}
\newcommand{\mges} {\sqsupseteq}
\newcommand{\segs} {\sqsubset} 
\newcommand{\sges} {\sqsupset} 
\newcommand{\tri} {\mathrel{\lhd}} 
\newcommand{\ClUp}[1] {\langle #1\rangle}

\newcommand{\Schreier}{\mathcal{S}}

\newcommand{\shift}[1]{{}_{*}#1}
\newcommand{\SHIFT}{\mathsf{S}}

\newcommand{\finSub}[1] {[#1]^{<\infty}}
\newcommand{\infSub}[1]{[#1]^{\infty}}

\newcommand{\Su}{\mathsf{s}}
\newcommand{\IIf}{\mathcal{E}}

\newcommand{\I}{\mathrm{I}}
\newcommand{\II}{\mathrm{II}}

\newcommand{\rtrans}[1]{\mathcal{R}_{#1}}

\title{Towards Better: a motivated introduction to better-quasi-orders}

\author{Yann Pequignot\thanks{The present article is adapted from the introductory chapter of the author's PhD thesis. He therefore would like to take the opportunity to thank the University of Lausanne for hospitality and support during the writing of his thesis. He also gratefully acknowledges the support of the Swiss National Science Foundation (SNF) through grant P2LAP2$\_$164904.}}

\begin{document}

\hyphenation{sub-sequence}

\renewcommand{\setminus}{\smallsetminus}
\renewcommand{\geq}{\geqslant}
\renewcommand{\leq}{\leqslant}
\renewcommand{\ngeq}{\ngeqslant} 
\renewcommand{\nleq}{\nleqslant}


\maketitle

\begin{abstract}
The well-quasi-orders (WQO) play an important role in various fields such as Computer Science, Logic or Graph Theory. Since the class of WQOs lacks closure under some important operations, the proof that a certain quasi-order is WQO consists often of proving it enjoys a stronger and more complicated property, namely that of being a better-quasi-order (BQO). 

Several articles -- notably \cite{milner1985basic,kruskal1972theory,simpson1985bqo,laverfraisse,laver1976well,forster2003better} -- contains valuable introductory material to the theory of BQOs. However, a textbook entitled \enquote{Introduction to better-quasi-order theory} is yet to be written. Here is an attempt to give a motivated and self-contained introduction to the deep concept defined by Nash-Williams that we would expect to find in such a textbook.
\end{abstract}

%

\section{Introduction}
Mathematicians have imagined a myriad of objects, most of them infinite, and inevitably followed by an infinite suite. 

What does it mean to understand them? 
How does a mathematician venture to make sense of these infinities he has imagined? 

Perhaps, one attempt could be to organise them, to arrange them, to order them.
At first, the mathematician can try to achieve this in a relative sense by comparing the objects according to some idea of complexity; 
this object should be above that other one, those two should be side by side, etc. 
So the graph theorist may consider the \emph{minor relation} between graphs, the recursion theorist may study the \emph{Turing reducibility} between sets of natural numbers, the descriptive set theorist can observe subsets of the Baire space through the lens of the \emph{Wadge reducibility} or equivalence relations through the prism of the \emph{Borel reducibility}, or the set theorist can organise ultrafilters according to the \emph{Rudin-Keisler ordering}.

This act of organising objects amounts to considering an instance of the very general mathematical notion of a \emph{quasi-order} (qo), namely a transitive and reflexive relation. 

As a means of classifying a family of objects, the following property of a quasi-order is usually desired: a quasi-order is said to be \emph{well-founded} if every non-empty sub-family of objects admits a minimal element. This means that there are minimal -- or simplest -- objects which we can display on a first bookshelf, and then, amongst the remaining objects there are again simplest objects which we can display on a second bookshelf above the previous one, and so on and so forth -- most probably into the transfinite.

However, as a matter of fact another concept has been \enquote{frequently discovered} \cite{kruskal1972theory} and proved even more relevant in diverse contexts: a \emph{well-quasi-order} (\wqo{}) is a well-founded quasi-order which contains no infinite antichain. Intuitively a well-quasi-order provides a satisfactory notion of hierarchy: as a well-founded quasi-order, it comes naturally equipped with an ordinal rank and there are up to equivalence only finitely many elements of any given rank. To prolong our metaphor, this means that, in particular, every bookshelf displays only finitely many objects -- up to equivalence.

The theory of \wqo{}s consists essentially of developing tools in order to show that certain quasi-orders suspected to be \wqo{} are indeed so. This theory exhibits a curious and interesting phenomenon: to prove that a certain quasi-order is \wqo{}, it may very well be easier to show that it enjoys a much stronger property. This observation may be seen as a motivation for considering the complicated but ingenious concept of \emph{better-quasi-order} (\bqo{}) invented by Crispin~St.~J.~A. \textcite{nash1965welltrees}. The concept of \bqo{} is weaker than that of well-ordered set but it is stronger than that of \wqo{}. In a sense, \wqo{} is defined by a single \enquote{condition}, while uncountably many \enquote{conditions} are necessary to characterise \bqo{}. Still, as Joseph B. \textcite[p.302]{kruskal1972theory} observed in 1972: \enquote{all \enquote{naturally occurring} \wqo{} sets which are known are \bqo{}}\footnote{The minor relations on finite graphs, proved to be \wqo{} by \textcite{Robertson2004325}, is to our knowledge the only naturally occurring \wqo{} which is not \emph{yet} known to be \bqo{}.}. 

\paragraph{Organisation of the paper}

In \cref{sec WQO} we give many characterisations of well-quasi-orders, all of them are folklore except maybe the one stated in \cref{prop WqoRado} which benefits from both an order-theoretical and a topological flavour. 

We make our way towards the definition of better-quasi-orders in \cref{sec BQO}. 
One of the difficulties we encountered when we began studying better-quasi-order is due to the existence of two main different definitions -- obviously equivalent to experts -- and along with them two different communities, the graph theorists and the descriptive set theorists, who only rarely cite each other in their contributions to the theory. The link between the original approach of Nash-Williams (graph theoretic) with that of Simpson (descriptive set theoretic) is merely mentioned by \textcite{argyros2005ramsey} alone. We present basic observations in order to remedy this situation in \cref{sec:MultisSeqSimpson}. Building on an idea due to \textcite{forster2003better}, we introduce the definition of better-quasi-order in a new way, using insight from one of the great contributions of descriptive set theory to better-quasi-order theory, namely the use of games and determinacy.

Finally in \cref{sec AroundBQO} we put the definition of better-quasi-order into perspective. This last section contains original material which have not been published elsewhere by the author.

\section{Well-quasi-orders}\label{sec WQO}

A reflexive and transitive binary relation $\leq$ on a set $Q$ is called a \emph{quasi-order} (qo, also preorder). As it is customary, we henceforth make an abuse of terminology and refer to the pair $(Q,\leq)$ simply as $Q$ when there is no danger of confusion. Moreover when it is necessary to prevent ambiguity we use a subscript and write $\leq_{Q}$ for the binary relation of the quasi-order $Q$.

The notion of quasi-order is certainly the most general mathematical concept of ordering. Two elements $p$ and $q$ of a quasi-order $Q$ are \emph{equivalent}, in symbols $p\equiv q$, if both $p\leq q$ and $q\leq p$ hold. It can very well happen that $p$ is equivalent to $q$ while $p$ is not equal to $q$. This kind of situation naturally arises when one considers for example the quasi-order of embeddability among a certain class of structures. Examples of pairs of structures which mutually embed into each other while being distinct, or even non isomorphic, abound in mathematics.

Every quasi-order has an \emph{associated strict relation}, denoted by $<$, defined by $p<q$ if and only if $p\leq q$ and $q\nleq p$ -- equivalently $p\leq q$ and $p\not\equiv q$. We say two elements $p$ and $q$ are \emph{incomparable}, when both $p\nleq q$ and $q\nleq p$ hold, in symbols $p\incomp q$.

A map $f:P\to Q$ between quasi-orders is \emph{order-preserving} (also isotone) if whenever $p\leq_{P} p'$ holds in $P$ we have $f(p)\leq_{Q} f(p')$ in $Q$. An \emph{embedding} is a map $f:P\to Q$ such that for every $p$ and $p'$ in $P$, $p\leq_{P} p'$ if and only if $f(p)\leq_{Q}f(p')$. Notice that an embedding is not necessarily injective.  An embedding $f:P\to Q$ is called an \emph{equivalence}\footnote{Viewing quasi-orders as categories in the obvious way, this notion of equivalence coincides with the one used in category theory.} provided it is \emph{essentially surjective}, i.e. for every $q\in Q$ there exists $p\in P$ with $q\equiv_{Q} f(p)$. We say that two quasi-orders $P$ and $Q$ are \emph{equivalent} if there exists an equivalence from $P$ to $Q$ -- by the axiom of choice this is easily seen to be an equivalence relation on the class of quasi-orders. Notice that every set $X$ quasi-ordered by the full relation $X\times X$ is equivalent to the one point quasi-order. In contrast, by an \emph{isomorphism} $f$ from $P$ to $Q$ we mean a bijective embedding $f:P\to Q$. 
Of course, a set $X$ quasi-ordered with the full relation $X\times X$ is never isomorphic to $1$ except when $X$ contains exactly one element.

In the sequel we study quasi-orders up to equivalence, namely only properties of quasi-orders which are preserved by equivalence are considered.

A quasi-order $Q$ is called a \emph{partial order} (po, also poset) provided the relation $\leq$ is antisymmetric, i.e. $p\equiv p$ implies $p=q$ -- equivalent elements are equal. Notice that an embedding between partial orders is necessarily injective. Moreover if $P$ and $Q$ are partial orders and $f:P\to Q$ is an equivalence, then $f$ is an isomorphism. We also note that in a partial order the associated strict order can also be defined by $p<q$ if and only if $p\leq q$ and $p\neq q$. 

Importantly, every quasi-order $Q$ admits up to isomorphism a unique equivalent partial order, its \emph{equivalent partial order}, which can be obtained as the quotient of $Q$ by the equivalence relation $p\equiv q$.

Even though most naturally occurring examples and constructions are only quasi-orders, one can always think of the equivalent partial order. The study of quasi-orders therefore really amounts to the study of partial orders.

\subsection{Good versus bad sequences}

We let $\omega=\{0,1,2,\ldots\}$ be the set of natural numbers. We use the set theoretic definitions $0=\emptyset$ and $n=\{0,\ldots,n-1\}$, so that the usual order on $\omega$ coincides with the membership relation. The equality and the usual order on $\omega$ give rise to the following distinguished types of sequences into a quasi-order.

\begin{definitions}Let $Q$ be a quasi-order.
\begin{enumDefn}
\item An \emph{infinite antichain} is a map $f:\omega \to Q$ such that for all $m,n\in\omega$, $m\neq n$ implies $f(m)\incomp f(n)$.
\item An \emph{infinite descending chain}, or an infinite decreasing sequence in $Q$, is a map $f:\omega \to Q$ such that for all $m,n\in\omega$, $m<n$ implies $f(m)>f(n)$.
\item A \emph{perfect sequence}, is a map $f:\omega\to Q$ such that for all $m, n\in\omega$ the relation $m\leq n$ implies $f(m)\leq f(n)$. In other words, $f$ is perfect if it is order-preserving from $(\omega,\leq)$ to $(Q,\leq)$.
\item A \emph{bad sequence} is a map $f:\omega\to Q$ such that for all $m, n\in\omega$, $m<n$ implies $f(m)\nleq f(n)$.
\item A \emph{good sequence} is a map $f:\omega\to Q$ such that there exist $m,n\in\omega$ with $m<n$ and $f(m)\leq f(n)$. Hence a sequence is good exactly when it is not bad. 
\end{enumDefn}
\end{definitions}

For any infinite subset $X$ of $\omega$, we denote by $[X]^{2}$ the set of pairs $\{x,y\}$ for distinct $x, y\in X$. When we write $\{m,n\}$ for a pair of natural numbers, we always assume it is written in increasing order ($m<n$). By Ramsey's Theorem\footnote{Nash-Williams' generalisation of Ramsey's Theorem is stated and proved as \cref{cor:NashWill}.} \cite{ramsey1930problem} whenever $[\omega]^{2}$ is partitioned into $P_{0}$ and $P_{1}$ there exists an infinite subset $X$ of $\omega$ such that either $[X]^{2}\subseteq P_{0}$, or $[X]^{2}\subseteq P_{1}$.

\begin{proposition}\label{prop:ClassicEquWqo1}
For a quasi-order $Q$, the following conditions are equivalent.
\begin{enumWqo}[series=Wqo]
\item \label{Item:1}$Q$ has no infinite descending chain and no infinite antichain;

\item \label{Item:2}there is no bad sequence in $Q$;

\item \label{Item:3}every sequence in $Q$ admits a perfect sub-sequence.
\end{enumWqo}
\end{proposition}
\begin{proof}
\begin{descProofUp}
\item[\cref{Item:1}$\to$\cref{Item:2}] We prove the contrapositive. Suppose that $f:\omega\to Q$ is a bad sequence. Partition $[\omega]^{2}$ into $P_{0}$ and $P_{1}$ with 
\[
P_{0}=\bigl\{\{m,n\}\in[\omega]^{2}\bigm\vert f(m)\ngeq f(n)\bigr\}.
\]
By Ramsey's Theorem, there exists an infinite subset $X$ of integers with either $[X]^{2}\subseteq P_{0}$, or $[X]^{2}\subseteq P_{1}$. In the first case $f:X\to Q$ is an infinite antichain and in the second case $f:X\to Q$ is an infinite descending chain.

\item[\cref{Item:2}$\to $\cref{Item:1}] Notice that an infinite antichain and an infinite descending chain are two examples of a bad sequence.

\item[\cref{Item:2}$\leftrightarrow $\cref{Item:3}] Let $f:\omega \to Q$ be any sequence in $Q$. 
We partition $[\omega]^{2}$ in $P_{0}$ and $P_{1}$ with 
\[
P_{0}=\bigl\{\{m,n\}\in[\omega]^{2}\bigm| f(m)\nleq f(n)\bigr\}.
\]
By Ramsey's Theorem, there exists an infinite subset $X$ of integers such that $[X]^{2}\subseteq P_{0}$ or $[X]^{2}\subseteq P_{1}$. The first case yields a bad sub-sequence. The second case gives a perfect sub-sequence. \qedhere
\end{descProofUp}
\end{proof}

\begin{definition}
A quasi-order $Q$ is called a \emph{well-quasi-order} (\wqo{}) when one of the equivalent conditions of the previous proposition is fulfilled. A quasi-order with no infinite descending chain is said to be \emph{well-founded}.
\end{definition}

The notion of \wqo{} is a frequently discovered concept, for an historical account of its early development we refer the reader to the excellent article by \textcite{kruskal1972theory}.

Using \cref{prop:ClassicEquWqo1} and the Ramsey's Theorem for pairs, one easily proves the following basic closure properties of the class of \wqo{}s.

\begin{proposition}\hfill \label{prop:ClosPropWQO}
\begin{enumThm}
\item If $(Q,\leq_{Q})$ is \wqo{} and $P\subseteq Q$, then $(P,\leq_{P})$ is \wqo{}, where $p\leq_{P} p'$ if and only if $p,p'\in P \text{ and }p\leq_{Q}p'$.
\item If $(P,\leq_{P})$ and $(Q,\leq_{Q})$ are \wqo{}, then $P\times Q$ quasi-ordered by 
\[
(p,q)\leq_{P\times Q} (p',q') \SSI p\leq_{P} p' \text{ and } q\leq_{Q} q'
\]
is \wqo{}.
\item If $(P,\leq_{P})$ is a \emph{partial order} and $(Q_{p},\leq_{Q_{p}})$ is a quasi-order for every $p\in P$, the sum $\sum_{p\in P}Q_{p}$ of the $Q_{p}$ along $P$ has underlying set the disjoint union $\{(p,q)\mid p\in P \text{ and }q\in Q_{p}\}$ and is quasi-ordered by 
\[
(p,q)\leq (p',q') \SSI \text{either } p=p' \text{ and } q\leq_{Q_{p}}q' \text{, or }
p<p'.
\]
If $P$ is \wqo{} and each $Q_{p}$ is \wqo{}, then $\sum_{p\in P}Q_{p}$ is \wqo{}.
\item If $Q$ is \wqo{} and there exists a map $g:P\to Q$ such that for all $p,p'\in P$ $g(p)\leq g(p') \to p\leq p'$, then $P$ is \wqo{}.
\item If $P$ is \wqo{} and there is a surjective and monotone map $h:P\to Q$, then $Q$ is \wqo{}.
\end{enumThm}
\end{proposition}

\subsection{Subsets and downsets}

Importantly, a \wqo{} can be characterised in terms of its subsets.

\begin{definitions}
Let $Q$ be a quasi-order.
\begin{enumDefn}
\item A subset $D$ of $Q$  is a \emph{downset}, or an \emph{initial segment}, if $q\in D$ and $p\leq q$ implies $p\in D$. For any $S\subseteq Q$, we write $\dcl S$ for the downset generated by $S$ in $Q$, i.e. the set $\{q\in Q\mid \exists p\in S \  q\leq p\}$. We also write $\dcl p$ for $\dcl\{p\}$.

We denote by $\Down(Q)$ the partial order of downsets of $Q$ under inclusion.

\item We give the dual meaning to \emph{upset}, $\upcl S$ and $\upcl q$ respectively. 

\item An upset $U$ is said to to be \emph{finitely generated}, or to admit a \emph{finite basis}, if there exists a finite $F\subseteq U$ such that $U=\upcl F$. We say that $Q$ has \emph{the finite basis property} if every upset of $Q$ admits a finite basis.

\item A downset $D\in \Down(Q)$ is said to be \emph{finitely bounded}, if there exists a finite set $F\subseteq Q$ with $D=Q\setminus \upcl F$. We let $\DownFb(Q)$ be the set of finitely bounded downsets partially ordered by inclusion.

\item We turn the power-set of $Q$, denoted $\mathcal{P}(Q)$, into a quasi-order by letting $X\leq Y$ if and only if $\forall p\in X\ \exists q\in Y\ p\leq q$, this is sometimes called the \emph{domination quasi-order}. We let $\CtblSubs(Q)$ be the the set of countable subsets of $Q$ with the quasi-order induced from $\mathcal{P}(Q)$. Since $X\leq Y$ if and only if $\dcl X \subseteq \dcl Y$, the equivalent partial order of $\mathcal{P}(Q)$ is $\Down(Q)$ and the quotient map is given by $X\mapsto \dcl X$.

\end{enumDefn}
\end{definitions}

The notion of well-quasi-order should be thought of as a generalisation of the notion of well-ordering beyond linear orders. Recall that a partial order $P$ is a linear order if for every $p$ and $q$ in $P$, either $p\leq q$ or $q\leq p$. A \emph{well-ordering} is (traditionally, the associated strict relation $<$ of) a partial order that is both linearly ordered and well-founded. 

Observe that a linearly ordered $P$ is well-founded if and only if the initial segments of $P$ are well-founded under inclusion. Considering for example the partial order $(\omega,=)$, one directly sees that a partial order $P$ can be well-founded while the initial segments of $P$ (here $\mathcal{P}(\omega)$) are not well-founded under inclusion. However a quasi-order $Q$ is \wqo{} if and only if the initial segments of $Q$ are well-founded under inclusion.

\begin{proposition}\label{prop:ClassicEquWqo2}
A quasi-order $Q$ is a \wqo{} if and only if one of the following equivalent conditions is fulfilled:
\begin{enumWqo}[Wqo]
\item \label{Item:4}$Q$ has the finite basis property,

\item \label{Item:5}$(\mathcal{P}(Q),\leq)$ is well-founded,

\item \label{Item:5b}$(\CtblSubs(Q),\leq)$ is well-founded,

\item \label{Item:6}$(\Down(Q),\subseteq)$ is well-founded,

\item \label{Item:6b} $(\DownFb(Q),\subseteq)$ is well-founded.
\end{enumWqo}
\end{proposition}

\begin{proof} 
\begin{descProofUp}
\item[\cref{Item:2}$ \to $\cref{Item:4}] We prove the contrapositive. Suppose $S\in\Final(Q)$ admits no finite basis. Since $\emptyset=\upcl \emptyset$, $S\neq \emptyset$. By dependent choice, we can show the existence of a bad sequence $f:\omega \to Q$. Choose $f(0)\in S$ and suppose that $f$ is defined up to some $n>0$. Since $\upcl \{f(0),\ldots f(n)\}\subset S$ we can choose some $f(n+1)$ inside $S\setminus \upcl \{f(0),\ldots f(n)\}$.

\item [\cref{Item:4}$ \to$\cref{Item:5}] We prove the contrapositive again. Suppose that $(X_{n})_{n\in\omega}$ is an infinite descending chain in $\mathcal{P}(Q)$. Then for each $n\in\omega$ we choose $q_{n}\in \dcl X_{n}\setminus \dcl X_{n+1}$. Then $\{q_{n}\mid n\in\omega\}$ has no finite basis. Indeed for all $n\in \omega$ we have $q_{n+1}\notin\upcl \{q_{i}\mid i\leq n\}$, otherwise $q_{i}\leq q_{n+1}\in \dcl X_{n+1}\subseteq \dcl X_{i+1}$ for some $i\leq n$, a contradiction.

\item[\cref{Item:5}$ \to $\cref{Item:5b}] Obvious.

\item[\cref{Item:5b}$ \to $\cref{Item:2}] By contraposition, if $(q_{n})_{n\in\omega}$ is a bad sequence in $Q$, then $P_{n}=\{q_{k}\mid n\leq k\}$ is an infinite descending chain in $\CtblSubs(Q)$ since whenever $m<n$ we have $q_{m}\in P_{m}$ and $q_{m}\nleq q_{k}$ for every $k\geq n$.

\item[\cref{Item:5}$ \to $\cref{Item:6}] By contraposition, any infinite descending chain for inclusion in $\Down(Q)$ is also an infinite descending chain in $(\mathcal{P}(Q),\leq)$.

\item[\cref{Item:6}$ \to $\cref{Item:6b}] Obvious.

\item[\cref{Item:6b}$\to$\cref{Item:2}] By contraposition, if $f:\omega\to Q$ is a bad sequence, then $n\mapsto D_{n}=Q\setminus \upcl \{f(i)\mid i\leq n\}$ is an infinite descending chain in $\DownFb(Q)$.\qedhere
\end{descProofUp}
\end{proof}

\subsection{Regular sequences}

A monotone decreasing sequence of ordinals is, by well-foundedness, eventually constant. The limit of such a sequence exists naturally, and is simply its minimum.

In general the limit of a sequence $(\alpha_i)_{i\in\omega}$ of ordinals may not exist, however any sequence of ordinals admits a limit superior. Indeed, define the sequence $\beta_i=\sup_{j\geq i}\alpha_j=\bigcup_{j\geq i} \alpha_{j}$, then $(\beta_i)_{i\in\omega}$ is decreasing and hence admits a limit.

We say that $(\alpha_i)_{i\in\omega}$ is \emph{regular} if the limit superior and the supremum $\bigcup_{i\in \omega}\alpha_{i}$ of $(\alpha_i)_{i\in\omega}$ coincide. This is equivalent to saying that for every $i\in \omega$ there exists $j>i$ with $\alpha_{i}\leq \alpha_{j}$. By induction one shows that this is in turn equivalent to saying that for all $i\in\omega$ the set $\{j\in\omega\mid i<j\text{ and } \alpha_{i}\leq\alpha_{j}\}$ is infinite.

\begin{notation}
For $n\in \omega$ and $X$ an infinite subset of $\omega$ let us denote by $X/n$ the final segment of $X$ given by $\{k\in X\mid k>n\}$. 
\end{notation}

We generalise the definition of regular sequences of ordinals to sequences in quasi-orders as follows.

\begin{definition}
Let $Q$ be a qo. A \emph{regular sequence} is a map $f:\omega\to Q$ such that for all $n\in \omega$ the set $\{k\in\omega/n\mid f(n)\leq f(k)\}$ is infinite.
\end{definition}

Here is a characterisation of \wqo{} in terms of regular sequences which exhibits another property of well-orders shared by \wqo{}s.
\begin{proposition}\label{prop: RegSubseq}
Let $Q$ be a qo. Then $Q$ is \wqo{} if and only if one of the following equivalent conditions holds:
\begin{enumWqo}[Wqo]
\item \label{Item:9}Every sequence in $Q$ admits a regular sub-sequence.
\item \label{Item:10}For every sequence $f:\omega\to Q$ there exists $n\in\omega$ such that the restriction $f:\omega/n\to Q$ is regular.
\end{enumWqo}
\end{proposition}
\begin{proof}
\begin{descProofUp}
\item [\cref{Item:6}$\to$\cref{Item:10}] For $f:\omega\to Q$ we let $f':\omega\to \Down(Q)$ be defined by $f'(n)=\dcl\{f(k)\mid n\leq k<\omega\}$. Then clearly if $m<n$ then $f'(m)\supseteq f'(n)$. The partial order $\Down(Q)$ being well-founded by \cref{Item:6}, there exists $n\in\omega$ such that for every $m>n$ we have $f'(n)=f'(m)$. This $n$ is as desired. Indeed, if $k>n$ then for every $l>k$ we have $f(k)\in f'(k)=f'(l)$ and so there exists $j\geq l$ with $f(k)\leq f(j)$. 

\item [\cref{Item:10}$\to$\cref{Item:9}] Obvious.

\item[\cref{Item:9}$\to$\cref{Item:2}] By contraposition, if $f:\omega\to Q$ is a bad sequence, then every sub-sequence of $f$ is bad. Clearly a bad sequence $f:\omega\to Q$ is not regular since for every $n\in\omega$ the set $\{k\in\omega/n\mid f(n)\leq f(k)\}$ is empty. Hence a bad sequence admits no regular sub-sequence. \qedhere
\end{descProofUp}
\end{proof}

\subsection{Sequences of subsets}\label{subsec SeqOfSubsets}

In this \lcnamecref{subsec SeqOfSubsets} we give a new characterisation of \wqo{}s which enjoys both a topological and an order-theoretical flavour. 

So far, we have considered $\Down(Q)$ as partially ordered set for inclusion. But $\Down(Q)$ also admits a natural topology which turns it into a compact Hausdorff $0$-dimensional space. Consider $Q$ as a discrete topological space, and form the product space $2^{Q}$, whose underlying set is identified with $\mathcal{P}(Q)$. This product space, sometimes called \emph{generalised Cantor space}, admits as a basis the clopen sets of the form
\[
N(F,G)=\{X\subseteq Q\mid F\subseteq X \text{ and } X\cap G=\emptyset\}, 
\]
for finite subsets $F,G$ of $Q$. For $q\in Q$, we write $\ClUp{q}$ instead of $N(\{q\},\emptyset)$ for the clopen set $\{X\subseteq Q\mid q\in X\}$.  Note that $\ClUp{q}^{\complement}=N(\emptyset,\{q\})$.

Notice that $\Down(Q)$ is an intersection of clopen sets,
\[\Down(Q)=\bigcap_{p\leq q} \ClUp{q}^{\complement} \cup \ClUp{p},\]
hence $\Down(Q)$ is closed in $2^{Q}$ and therefore compact.

Now recall that for every sequence $(E_{n})_{n\in\omega}$ of subsets of $Q$ we have the usual relations
\begin{equation}\label{eq IntLimInfLimSupUnion}
\bigcap_{n\in\omega} E_{n}\subseteq \bigcup_{i\in \omega}\bigcap_{j\geq i}E_{j}\subseteq \bigcap_{i\in \omega}\bigcup_{j\geq i}E_{j}\subseteq \bigcup_{n\in \omega} E_{n}.
\end{equation}

Moreover the convergence of sequences in $2^{Q}$ can be expressed by means of a \enquote{$\liminf=\limsup$} property. 

\begin{fact}\label{fac liminf=limsup}
A sequence $(E_{n})_{n\in\omega}$ converges to $E$ in $2^{Q}$ if and only if
\[\bigcup_{i\in \omega}\bigcap_{j\geq i}E_{j}=\bigcap_{i\in \omega}\bigcup_{j\geq i}E_{j}=E.\]
\end{fact}
\begin{proof}
Suppose that $E=\bigcup_{i\in \omega}\bigcap_{j\geq i}E_{j}=\bigcap_{i\in \omega}\bigcup_{j\geq i}E_{j}$. We show that $E_{n}\to E$. Let $F,G$ be finite subsets of $Q$ with $E\in N(F,G)$. Since $E=\bigcup_{i\in \omega}\bigcap_{j\geq i}E_{j}$ and $F$ finite, $F\subseteq E_{j}$ for all sufficiently large $j$. Since $E=\bigcap_{i\in \omega}\bigcup_{j\geq i}E_{j}$ and $G$ is finite, $G\cap E_{j}=\emptyset$ for all sufficiently large $j$. It follows that $E_{j}\in N(F,G)$ for all sufficiently large $j$, whence $(E_{n})_{n}$ converges to $E$.

Conversely, assume that $E_{n}$ converges to some $E$ in $2^{Q}$. If $q$ belongs to $E$ -- i.e. $E\in \ClUp{q}$ -- then $q\in E_{j}$ for all sufficiently large $j$ and thus $q\in \bigcup_{i\in \omega}\bigcap_{j\geq i}E_{j}$. And if $q\notin E$, i.e. $E\notin \ClUp{q}$, then $q\notin E_{j}$ for all sufficiently large $j$ and thus $q\notin \bigcap_{i\in\omega}\bigcup_{j\geq i}E_{j}$. Therefore by \cref{eq IntLimInfLimSupUnion} it follows that $E=\bigcup_{i\in \omega}\bigcap_{j\geq i}E_{j}=\bigcap_{i\in \omega}\bigcup_{j\geq i}E_{j}$.
\end{proof}

Observe that if $(q_{n})_{n\in\omega}$ is a perfect sequence in a qo $Q$, then for every $q\in Q$ if $q \leq q_{m}$ for some $m$, then $q\leq q_{n}$ holds for all $n\geq m$. Therefore by \cref{eq IntLimInfLimSupUnion} we have
\[
\bigcup_{m\in\omega}\bigcap_{n\geq m} \dcl q_{n}=\bigcup_{n\in\omega}\dcl q_{n},
\]
whence $(\dcl q_{n})_{n\in\omega}$ converges to $\dcl \{q_{n}\mid n\in\omega\}$ in $2^{Q}$ by \cref{fac liminf=limsup}. On the contrary no bad sequence $(q_{n})_{n\in\omega}$ converges towards $\dcl\{q_{n}\mid n\in\omega\}$, since for example $q_{0}$ does not belong to $\bigcup_{i\in \omega}\bigcap_{j\geq i}\dcl q_{j}$. We have obtained the following:  

\begin{fact}\label{fac RadoPrincipWqo}
Let $Q$ be a qo. 
\begin{enumThm}
\item $Q$ is \wqo{} if and only if for every sequence $(q_{n})_{n\in\omega}$ there exists $N\in\infSub{\omega}$ such that $(\dcl q_{n})_{n\in N}$ converges to $\dcl\{q_{n}\mid n\in N\}$ in $\Down(Q)$.
\item If $Q$ is \wqo{} and $(\dcl q_{n})_{n\in\omega}$ converges to some $D$ in $\Down(Q)$, then there is some $N\in\infSub{\omega}$ such that $D=\dcl \{q_{n}\mid n\in N\}$.
\end{enumThm}
\end{fact}

Actually more is true, thanks to the following ingenious observation made by Richard Rado in the body of a proof in \parencite{rado1954partial}.

\begin{lemma}[Rado's trick]\label{radoLem}
Let $Q$ be a \wqo{} and let $(D_{n})_{n\in\omega}$ be a sequence in $\Down(Q)$. Then there exists an infinite subset $N$ of $\omega$ such that
\[
\bigcup_{i\in N}\bigcap_{j\in N/i}D_{j}=\bigcup_{n\in N} D_{n},
\]
and so the sub-sequence $(D_{j})_{j\in N}$ converges to $\bigcup_{n\in N}D_{n}$ in $\Down(Q)$.
\end{lemma}\begin{proof}
Towards a contradiction suppose that for all infinite $N\subseteq \omega$ we have
\begin{equation}\label{eq Rado'sTrick}
\bigcup_{i\in N}\bigcap_{ j\in N/i}D_{j}\subset \bigcup_{n\in N} D_{n}.
\end{equation}
We define an infinite descending chain $(E_i)_{i\in\omega}$ in $\Down(Q)$. But to do so we recursively define a sequence $(N_k)_{k\in\omega}$ of infinite subsets of $\omega$ and a sequence $(q_{k})_{k\in \omega}$ in $Q$ such that
\begin{enumProof}
\item $N_{0}=\omega$ and $N_{k}\supseteq N_{k+1}$ for all $k\in \omega$.
\item $q_{k}\in \bigcup_{j\in N_{k}}D_{j}$ and $q_{k}\notin \bigcup_{j\in N_{k+1}} D_{j}$.
\end{enumProof}
Suppose we have defined $N_{0},\ldots,N_{k}$ and $q_{0},\ldots q_{k-1}$. By \cref{eq Rado'sTrick} we have 
\[
\bigcup_{n\in N_{k}}D_{n}\nsubseteq \bigcup_{i\in N_{k}}\bigcap_{j\in N_{k}/i} D_{j},\]
so we can pick $n_{0}\in N_{k}$ and $q_{k}\in D_{n_{0}}$ such that $q_{k}\notin \bigcup_{i\in N_{k}}\bigcap_{j\in N_{k}/i}D_{j}$. Then for all $i$ in $N_{k}$ let $j_{i}\in N_{k}/i$ be minimal such that $q_{k}\notin D_{j_{i}}$. Setting $n_{1}=j_{n_{0}}$ and $n_{i+1}=j_{n_{i}}$, we obtain an infinite set $N_{k+1}=\{n_{1},n_{2},\ldots\}$ which satisfies 
\[q_{k}\in D_{n_{0}}\subseteq \bigcup_{j\in N_{k}}D_{j} \quad \text{and}\quad q_{k}\notin \bigcup_{j\in N_{k+1}} D_{j}.\]

Now we define $E_{k}=\bigcup_{j\in N_{k}}D_{j}$. The sequence $(E_k)_{k\in\omega}$ is an infinite descending chain in $\Down(Q)$, contradicting the fact that $Q$ is \wqo{}.

For the second statement, observe that if $N$ is an infinite subset of $\omega$ satisfying the statement of the lemma, then by \cref{eq IntLimInfLimSupUnion} we have
\[
\bigcup_{i\in N}\bigcap_{j\in N/i}D_{j}= \bigcap_{i\in N}\bigcup_{j\in N/i}D_{j}= \bigcup_{n\in N} D_{n},
\]
and so by \cref{fac liminf=limsup} we get that $(D_{j})_{j\in N}$ converges to $\bigcup_{n\in N} D_{n}$ in $\Down(Q)$.
\end{proof}

Hence if $Q$ is \wqo{}, then every sequence in $\Down(Q)$ admits a sub-sequence which converges to its union. Of course the converse also holds.
 
\begin{lemma}
If $(D_{n})_{n\in\omega}$ is an infinite descending chain in $\Down(Q)$, then there is no infinite subset $N$ of $\omega$ such that $\bigcup_{n\in N} D_{n}=\bigcup_{i\in N}\bigcap_{j\in N/i}D_{j}$.
\end{lemma}
\begin{proof}
Since any sub-sequence of an infinite descending chain is again an infinite descending chain, it is enough to show that if $(D_{n})_{n\in\omega}$ is an infinite descending chain in $\Down(Q)$ then $\bigcup_{n\in \omega} D_{n} \nsubseteq \bigcup_{i\in \omega}\bigcap_{j\in \omega/i}D_{j}$. Pick any $q\in D_{0}\setminus D_{1}$. Then since $D_{j}\subseteq D_{1}$ for all $j\geq 1$ and $D_{1}$ is a downset, we get $q\notin D_{j}$ for all $j\geq 1$. It follows that $q\notin \bigcup_{i\in \omega}\bigcap_{j\in \omega/i}D_{j}$.
\end{proof}

This leads to our last characterisation of \wqo{}:
\begin{proposition}\label{prop WqoRado}
Let $Q$ be a qo. Then $Q$ is \wqo{} if and only if
\begin{enumWqo}[Wqo]
\item Every sequence $(D_{n})_{n\in\omega}$ in $\Down(Q)$ admits a sub-sequence $(D_{n})_{n\in N}$ which converges to $\bigcup_{n\in N}D_{n}$.
\end{enumWqo}
\end{proposition}

\section{Better-quasi-orders}\label{sec BQO}

\subsection{Towards better}\label{subsec TowardsBetter}

As we have seen in \cref{prop:ClassicEquWqo2} a quasi-order is \wqo{} if and only if $\mathcal{P}(Q)$ is well-founded if and only if $\Down(Q)$ is well founded. The first example of a \wqo{} whose powerset contains an infinite antichain was identified by Richard Rado. This \wqo{} is the starting point of the journey towards the stronger notion of better-quasi-order.

\begin{example}[\parencite{rado1954partial}]\label{ex:RadoPoset}
Rado's partial order $\Rado$ is the set $[\omega]^{2}$, of pairs of natural numbers, partially ordered by (cf. \cref{figRadoPoset}):
\[
\{m,n\}\leq \{m',n'\} \SSI \begin{cases}
\text{$m=m'$ and $n\leq n'$, or}\\
n<m'.
\end{cases} 
\]

\begin{figure}

\begin{center}
\includegraphics{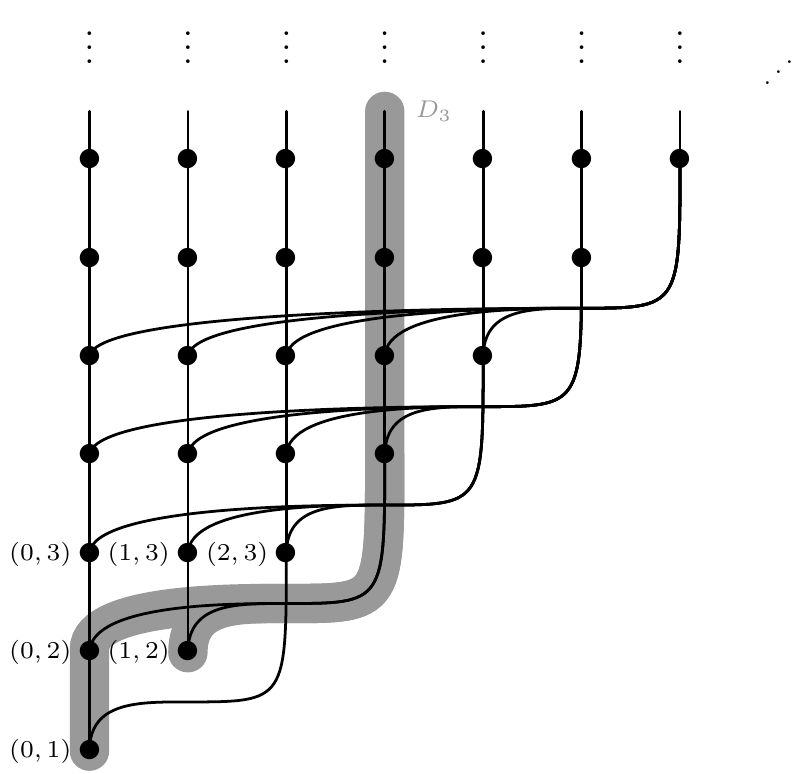}
\end{center}
\caption{Rado's poset $\Rado$.}
\label{figRadoPoset}

\end{figure}

The po $\Rado$ is \wqo{}. To see this, consider any map $f:\omega\to  [\omega]^{2}$ and let $f(n)=\{f_{0}(n),f_{1}(n)\}$ for all $n\in\omega$. Now if $f_{0}$ is unbounded, then there exists $n>0$ with $f_{1}(0)<f_{0}(n)$ and so $f(0)\leq f(n)$ in $\Rado$ by the second clause.
If $f_{0}$ is bounded, let us assume by going to a sub-sequence if necessary, that $f_{1}:\omega\to \omega$ is perfect. Then there exist $m$ and $n$ with $m<n$ and $f_{0}(m)=f_{0}(n)$ and we have $f_{1}(m)\leq f_{1}(n)$, so $f(m)\leq f(n)$ in $\Rado$ by the first clause. In both cases we find that $f$ is good, so $\Rado$ is \wqo{}.

However the map $n\mapsto D_{n}=\dcl\{\{n,l\}\mid n<l\}$ is a bad sequence (in fact an infinite antichain) inside $\Down(\Rado)$. Indeed whenever $m<n$ we have $\{m,n\}\in D_{m}$ while $\{m,n\}\notin D_{n}$, and so $D_{m}\nsubseteq D_{n}$. 
\end{example}

One natural question is now: What witnesses in a given quasi-order $Q$ the fact that $\mathcal{P}(Q)$ is not \wqo{}? It cannot always be a bad sequence, that is what the existence of Rado's partial order tells us. But then what is it?

To see this suppose that $(P_{n})_{n\in\omega}$ is a bad sequence in $\mathcal{P}(Q)$. Fix some $m\in \omega$. Then whenever $m<n$ we have $P_{m}\nsubseteq \dcl P_{n}$ and we can choose a witness $q\in P_{m}\setminus \dcl P_{n}$. But of course in general there is no single $q\in P_{m}$ that witnesses $P_{m}\nsubseteq \dcl P_{n}$ for all $n>m$. So we are forced to pick a sequence $f_{m}:\omega/m\to Q$, $n\mapsto q^{n}_{m}$ of witnesses:
\[
q^{n}_{m}\in P_{m} \quad \text{and}\quad q^{n}_{m}\notin \dcl P_{n}, \qquad \text{$n\in\omega/m$}.
\]
Bringing together all the sequences $f_{0},f_{1},\ldots$, we obtain a \emph{sequence of sequences}, naturally indexed by the set $[\omega]^{2}$ of pairs of natural numbers, 
\begin{align*}
f:[\omega]^{2}&\longrightarrow Q\\ 
\{m,n\}&\longmapsto f_{m}(n)=q^{n}_{m}.
\end{align*}
By our choices this sequence of sequences satisfies the following condition: 
\[
\forall m,n,l\in\omega \quad m<n<l \to q^{n}_{m}\nleq q^{l}_{n}.
\]
Indeed, suppose towards a contradiction that for $m<n<l$ we have $q_{m}^{n}\leq q_{n}^{l}$. Since $q_{n}^{l}\in P_{n}$ we would have $q_{m}^{n}\in \dcl P_{n}$, but we chose $q_{m}^{n}$ such that $q_{m}^{n}\notin \dcl P_{n}$.

Let us say that a sequence of sequences $f:[\omega]^{2}\to Q$ is \emph{bad} if for every $m,n,l\in\omega$, $m<n<l$ implies $f(\{m,n\})\nleq f(\{n,l\})$. We have come to the following.

\begin{proposition}\label{prop2BQOssiDownQwqo}
Let $Q$ be a qo. Then $\mathcal{P}(Q)$ is \wqo{} if and only if there is no bad sequence of sequences into $Q$.
\end{proposition}
\begin{proof}
As we have seen in the preceding discussion, if $\mathcal{P}(Q)$ is not \wqo{} then from a bad sequence in $\mathcal{P}(Q)$ we can make choices in order to define a bad sequences of sequences in $Q$.

Conversely, if $f:[\omega]^{2}\to Q$ is a bad sequence of sequences, then for each $m\in\omega$ we can consider the set $P_{m}=\{f(\{m,n\})\mid n\in\omega/m\}$ consisting in the image of the $m$\ts{th} sequence. Then the sequence $m\mapsto P_{m}$ in $\mathcal{P}(Q)$ is a bad sequence. Indeed every time $m<n$ we have $f(\{m,n\})\in P_{m}$ while $f(\{m,n\})\notin \dcl P_{n}$, since otherwise there would exist $l>n$ with $f(\{m,n\})\leq f(\{n,l\})$, a contradiction with the fact that $f$ is a bad sequence of sequences. 
\end{proof}

One should notice that from the previous proof we actually get that $\mathcal{P}(Q)$ is \wqo{} if and only if $\CtblSubs(Q)$ is \wqo{}.  Notice also that in the case of Rado's partial order $\Rado$ the fact that $\mathcal{P}(\Rado)$ is not \wqo{} is witnessed by the bad sequence $f:[\omega]^{2}\to\Rado, \{m,n\}\mapsto \{m,n\}$ which is simply the identity on the underlying sets, since every time $m<n<l$ then $\{m,n\}\nleq \{n,l\}$ in $\Rado$. In fact, Rado's partial order is in a sense universal as established by Richard Laver \cite{laver1976well}: 

\begin{theorem}
If $Q$ is \wqo{} but $\mathcal{P}(Q)$ is not \wqo{}, then $\Rado$ embeds into $Q$.
\end{theorem}
\begin{proof}
Let $f:[\omega]^{2}\to Q$ be a bad sequence of sequences. Partitioning the triples $\{i,j,k\}$, $i<j<k$, into two sets depending on whether or not $f(\{i,j\})\leq f(\{i,k\})$, we get by Ramsey's Theorem an infinite set $N\subseteq\omega$ whose triples are all contained into one of the classes. If for every $\{i,j,k\}\subseteq N$ we have $f(\{i,j\})\nleq f(\{i,k\})$ then for any $i\in N$ the sequence $f(\{i,j\})_{j\in N/i}$ is a bad sequence in $Q$. Since $Q$ is \wqo{}, the other possibility must hold. 

Then partition the quadruples $\{i,j,k,l\}$ in $N$ into two sets according to whether or not $f(\{i,j\})\leq f(\{k,l\})$. Again there exists an infinite subset $M$ of $N$ whose quadruples are all contained into one of the classes. If all quadruples $\{i,j,k,l\}$ in $M$ satisfy $f(\{i,j\})\nleq f(\{k,l\})$, then for any sequence $(\{i_{k},j_{k}\})_{k\in\omega}$ of pairs in $M$ with $j_{k}<i_{k+1}$ the sequence $f(\{i_{k},j_{k}\})_{k\in\omega}$ is bad in $Q$. Since $Q$ is \wqo{}, it must be the other possibility that holds. 

Let $X=M\setminus \{\min M\}$, then $\{f(\{i,j\})\mid \{i,j\}\in [X]^{2}\}$ is isomorphic to $\Rado$. By the properties of $M$, we have $\{i,j\}\leq\{k,l\}$ in $\Rado$ implies $f(\{i,j\})\leq f(\{k,l\})$. We show that $f(\{i,j\}) \leq f(\{k,l\})$ implies $\{i,j\}\leq \{k,l\}$ in $\Rado$. Suppose $\{i,j\}\nleq \{k,l\}$ in $\Rado$, namely $k\leq j$ and either $i\neq k$, or $l< j$. If $l<j$ and $f(\{i,j\})\leq f(\{k,l\})$ then for any $n\in X/j$ we have $f(\{k,l\})\leq f(\{j,n\})$ and thus $f(\{i,j\})\leq f(\{j,n\})$ a contradiction since $f$ is bad. Suppose now that $k\leq j$ and $i\neq k$. If $i<k$ and $f(\{i,j\})\leq f(\{k,l\})$, then $f(\{i,k\})\leq f(\{i,j\})\leq f(\{k,l\})$, a contradiction. Finally if $k<i$ and $f(\{i,j\})\leq f(\{k,l\})$ then for $m=\min M$ we have $f(\{m,k\})\leq f(\{i,j\})\leq f(\{k,l\})$, again a contradiction.
\end{proof}

From a heuristic viewpoint, a \emph{better-quasi-order} is a well quasi-order $Q$ such that $\mathcal{P}(Q)$ is \wqo{}, $\mathcal{P}(\mathcal{P}(Q))$ is \wqo{}, $\mathcal{P}(\mathcal{P}(\mathcal{P}(Q)))$ is \wqo{}, so on and so forth, into the transfinite. This idea will be made precise in \cref{sec bqoPower}, but we can already see that it cannot serve as a convenient definition\footnote{The reader who remains unconvinced can try to prove that the partial order $(3,=)$ satisfies this property.}. As the above discussion suggests, a better-quasi-order is going to be a qo $Q$, with no bad sequence, with no bad sequence of sequences, no bad sequence of sequences of sequences, so on and so forth, into the transfinite. To do so we need a convenient notion of \enquote{index set} for a sequence of sequences of~\dots\ of sequences, in short a \emph{super-sequence}. We now turn to the study of this fundamental notion defined by Nash-Williams.

\subsection{Super-sequences}

Let us first introduce some useful notation.  Given an infinite subset $X$ of $\omega$ and a natural number $k$, we denote by $[X]^{k}$ the set of subsets of $X$ of cardinality $k$, and by $\finSub{X}$ the set $\bigcup_{k\in\omega}[X]^{k}$ of finite subsets of $X$. When we write an element $s\in [X]^{k}$ as $\{n_{0},\ldots,n_{k-1}\}$ we always assume it is written in increasing order $n_{0}<n_{1}<\ldots< n_{k-1}$ for the usual order on $\omega$. The cardinality of $s\in\finSub{\omega}$ is denoted by $|s|$. We write $\infSub{X}$ for the set of infinite subsets of $X$.

For any $X\in \infSub{ \omega}$ and any $s\in\finSub{\omega}$, we let $X/s=\{k\in X\mid \max s<k\}$ and we write $X/n$ for $X/\{n\}$, as we have already done.

\subsubsection{Index sets for super-sequences}

Intuitively super-sequences are sequences of sequences~\dots\ of sequences. In order to deal properly with this idea we need a convenient notion of index sets. Those will be families of finite sets of natural numbers called \emph{fronts}. They were defined by \textcite{nash1965welltrees}. As the presence of an ellipsis in the expression \enquote{sequences of sequences of~\dots\ of sequences} suggests, the notion of front admits an inductive definition. To formulate such a definition it is useful to identify the degenerate case of a super-sequence, the level zero of the notion of sequence of~\dots\ of sequences, namely a function $f:1\to E$ which singles out a point of a set $E$. The index set for these degenerate sequences is the family $\{\emptyset\}$ called the \emph{trivial front}. New fronts are then built up from old ones using the following operation. 

\begin{definition}\label{def OperationSeq}
If $X\in\infSub{\omega}$ and $F(n)\subseteq \finSub{X/n}$ for every $n\in X$, we let
\[
\Seq_{n\in X} F(n)=\big\{\{n\}\cup s\mid n\in X \text{ and }s\in F(n)\big\}.
\]
\end{definition}

\begin{definition}[Front, inductive definition]
 We define a \emph{front on $X$} simultaneously for every $X\in\infSub{\omega}$ by induction using the two following clauses: 
\begin{enumDefn}
\item for all $X\in \infSub{\omega}$, the family $\{\emptyset\}$ is a front on $X$,
\item if $X\in \infSub{\omega}$ and if $F(n)$ is a front on $X/n$ for all $n\in X$, then 
\[
F=\Seq_{n\in X} F(n)
\]
is a front on $X$.
\end{enumDefn}
\end{definition}

\begin{remark}
In the literature, fronts are sometimes called \emph{blocks} or \emph{thin blocks}. Here we follow the terminology of \textcite{todorvcevic2010introduction}.
\end{remark}

\begin{examples}
We have already seen example of fronts. Indeed for every $X\in\infSub{\omega}$ and every $n\in\omega$ the family $[X]^{n}$ is a front on $X$, where $[X]^{0}=\{\emptyset\}$ is the trivial front. For a new example, consider for every $n\in \omega$ the front $[\omega/n]^{n}$ and build
\[
\Schreier=\Seq_{n\in \omega} \left[\omega/n\right]^{n}
=\{s\in\finSub{\omega}\mid 1+\min s=|s|\}.
\]
The front $\Schreier$ is traditionally called the \emph{Schreier barrier}.

\begin{figure}[ht]

\includegraphics{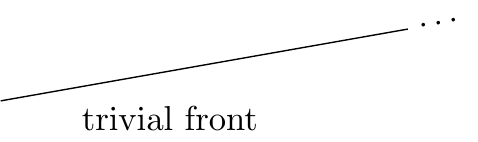} \quad \includegraphics{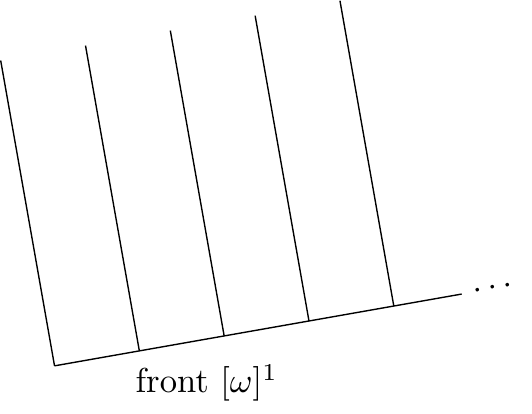}\\
\includegraphics{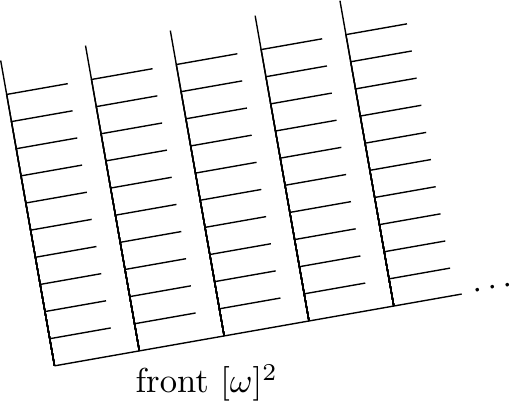} \quad \includegraphics{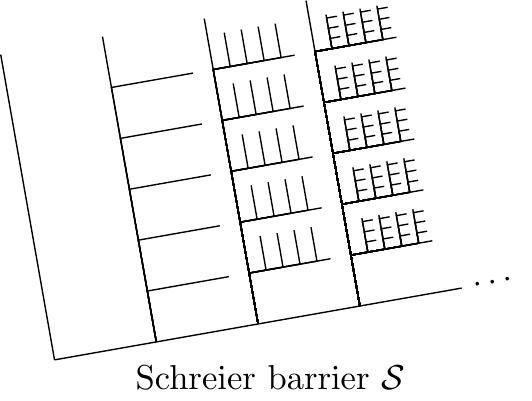}
\caption{Pictures of fronts}\label{FigFronts}
\end{figure}
\end{examples}

We defined fronts in order to make the following definition.
\begin{definition}
A \emph{super-sequence} in a set $E$ is a map $f:F\to E$ from a front into $E$.
\end{definition}

Notice that if $F$ is a non trivial front on $X$, we can recover the unique sequence $F(n)$, $n\in X$, of fronts from which it is constructed. 
\begin{definition}
For any family $F\subseteq \finSub{\omega}$ and $n\in\omega$ we define the \emph{ray} of $F$ at $n$ to be the family
\[
F_{n}=\bigl\{s\in\finSub{\omega/n}\bigm\vert \{n\}\cup s\in F\bigr\}.
\]
\end{definition}

Then every non trivial front $F$ on $X$ is built up from its rays $F_{n}$, $n\in X$, in the sense that:
\[
F=\Seq_{n\in X}F_{n}.
\]

Notice that, according to our definition, the trivial front $\{\emptyset\}$ is a front on $X$ for every $X\in\infSub{\omega}$. Except for this degenerate example, if a family $F\subseteq \finSub{X}$ is a front on $X$, then necessarily $X$ is equal to $\base F$, the set-theoretic union of the family $F$. For this reason we will sometimes say that $F$ is a front, without reference to any infinite subset $X$ of $\omega$. Moreover when $F$ is not trivial, we refer to the unique $X$ for which $F$ is a front on $X$, namely $\base F$, as the \emph{base} of $F$.

Importantly, the notion of a front also admits an explicit definition to which we now turn. It makes essential use of the following binary relation.

\begin{definition}
For subsets $u,v$ of $\omega$, we write $u\segm v$ when $u$ is an \emph{initial segment} of $v$, i.e. when $u=v$ or when there exists $n\in v$ such that $u=\{k\in v\mid k<n\}$. As usual, we write $u\segs v$ for $u\segm v$ and $u\neq v$. 
\end{definition}

\begin{definition}[Front, explicit definition]
A family $F\subseteq\finSub{\omega}$ is a \emph{front on} $X\in \infSub{\omega}$ if 
\begin{enumDefn}
\item either $F=\{\emptyset\}$, or $\base F=X$,\label{def Front1}
\item for all $s,t\in F$ $s\segm t$ implies $s=t$, \label{def Front2}
\item (Density) for all $X'\in\infSub{X}$ there is an $s\in F$ such that $s\segs X'$. \label{def Front3}
\end{enumDefn}
\end{definition}

Merely for the purpose of showing that our two definitions coincide, and only until this is achieved, let us refer to a front according to the explicit definition as a \front{}. Notice that the family $\{\emptyset\}$ is a \front{}, the trivial \front{}.  Notice also that if $F$ is a non trivial \front{} then necessarily $\emptyset\notin F$.

Our first step towards proving the equivalence of our two definitions of fronts is the following easy observation.

\begin{lemma}\label{lemFrontRays}
Let $F$ be a non trivial \front{} on $X\in \infSub{\omega}$. Then for every $n\in X$, the ray $F_{n}$ is a \front{} on $X/n$. Moreover $F=\Seq_{n\in X}F_{n}$. 
\end{lemma}
\begin{proof}
Let $n\in X$. For every $Y\in\infSub{X/n}$ there exists $s\in F$ with $s\segs \{n\}\cup Y$. Since $F$ is non trivial, $s\neq\emptyset$ and so $n\in s$. Therefore $s'=s\setminus \{n\}\in F_{n}$ with $s'\segs Y$, and $F_{n}$ satisfies \cref{def Front3}. Now if $F_{n}$ is not trivial and $k\in X/n$, there is $s\in F_{n}$ with $s\segs \{k\}\cup X/k$ and necessarily $k\in s\subseteq \base F_{n}$. Hence $\base F_{n}=X/n$, so condition \cref{def Front1} is met. To see \cref{def Front2}, let $s,t\in F_{n}$ with $s\segm t$. Then for $s'=\{n\}\cup s$ and $t'=\{n\}\cup t$ we have $s',t'\in F$ and $s'\segm t'$, so $s'=t'$ and $s=t$, as desired. The last statement is obvious.
\end{proof}

Our next step consists in assigning a rank to every \front{}. To do so, we first recall some classical notions about sequences and trees.

\begin{notation}
For a non empty set $A$, we write $A^{n}$ for the set of sequences $s:n\to A$. Let $A^{<\omega}$ be the set $\bigcup_{n\in\omega} A^{n}$ of finite sequences in $A$. We write $A^{\omega}$ for the set of infinite sequences $x:\omega \to A$ in $A$. Let $u\in A^{<\omega}$, $x\in A^{<\omega}\cup A^{\omega}$.
\begin{enumDefn}
\item $|x|\in \omega+1$ denotes the length of $x$.
\item For $n\leq |x|$, $x\restr{n}$ is the initial segment, or prefix, of $x$ of length $n$.
\item  We write $u\segm x$ if there exists $n\leq |x|$ with $u=x\restr{n}$. We write $u\segs x$ if $u\segm x$ and $u\neq x$.
\item We write $u\conc x$ for the concatenation operation.
\end{enumDefn}
\end{notation}

Identifying any finite subset of $\omega$ with its increasing enumeration with respect to the usual order on $\omega$, we view any \front{} as a subset of $\omega^{<\omega}$. 
Notice that under this identification, our previous definition of $\segm$ for subsets of $\omega$ coincides with the one for sequences. 

\begin{definitions}\label{def trees}
\begin{enumDefn}
\item A \emph{tree} $T$ on a set $A$ is a subset of $A^{<\omega}$ that is closed under prefixes, i.e. $u\segm v$ and $v\in T$ implies $u\in T$. 
\item A tree $T$ on $A$ is called \emph{well-founded} if $T$ has no infinite branch, i.e. if there is no infinite sequence $x\in A^{\omega}$ such that $x\restr{n}\in T$ holds for all $n\in\omega$. In other words, a tree $T$ is well-founded if $(T,\mges)$ is a well-founded partial order.

\item When $T$ is a non-empty well-founded tree we can define a strictly decreasing function $\rho_{T}$ from $T$ to the ordinals by transfinite recursion on the well-founded relation $\sges$:
\[
\rho_{T}(t)=\sup\{\rho_{T}(s)+1 \mid t\segs s\in T\} \quad\text{for all $t\in T$.}
\] 
It is easily shown to be equivalent to
\[
\rho_{T}(t)=\sup\{\rho_{T}(t\conc (a))+1\mid a\in A \text{ and } t\conc(a)\in T\} \quad\text{for all $t\in T$.}
\]
The \emph{rank} of the non-empty well-founded tree $T$ is the ordinal $\rho_{T}(\emptyset)$. 
\end{enumDefn}
\end{definitions}

For any \front{} $F$, we let $T(F)$ be the smallest tree on $\omega$ containing $F$, i.e.
\[
T(F)=\{s\in\omega^{<\omega}\mid \exists t\in F \ s\segm t\}.
\]

The following is a direct consequence of the explicit definition of a front.
\begin{lemma}\label{lem TreeFrontWellFounded}
For every \front{} $F$, the tree $T(F)$ is well-founded.
\end{lemma}
\begin{proof}
If $x$ is an infinite branch of $T(F)$, then $x$ enumerates an infinite subset $X$ of $\base F$ such that for every $u\segs X$ there exists $t\in F$ with $u\segm t$. Since $F$ is a \front{} there exists a (unique) $s\in F$ with $s\segs X$. Let $n=\min X/ s$ and for $u=s\cup \{n\}$ consider some $t\in F$ with $u\segm t$.
But then $F\ni s\segs u\segm t\in F$ which contradicts the explicit definition of a front. 
\end{proof}

\begin{definition}
Let $F$ be a \front{}. 
The \emph{rank} of $F$, denoted by $\rk F$, is the rank of the tree $T(F)$.
\end{definition}

\begin{example}
Notice that the family $\{\emptyset\}$ is the only \front{} of null rank, and for all positive integer $n$, the front $[\omega]^{n}$ has rank $n$. Moreover the Schreier barrier $\Schreier$ has rank $\omega$.
\end{example}

We now observe that the rank of $F$ is closely related to the rank of its rays $F_{n}$, $n\in X$. Let $F$ be a non trivial \front{} on $X\in \infSub{\omega}$ and recall that by \cref{lemFrontRays}, the ray $F_{n}$ is a \front{} on $X/n$ for every $n\in X$. Now notice that the tree $T(F_{n})$ of the \front{} $F_{n}$ is naturally isomorphic to the subset 
\[
\{s\in T(F)\mid \{n\}\segm s\} 
\] 
of $T(F)$. The rank of the \front{} $F$ is therefore related to the ranks of its rays through the following formula:
\[
\rk F=\sup \{\rk (F_{n})+1 \mid n\in X\}.
\]
In particular, $\rk F_{n}<\rk F$ for all $n\in X$.

This simple remark allows one to prove results on \front{}s by induction on the rank by applying the induction hypothesis to the rays, as it was first done by \textcite{pudlak1982partition}. It also allows us to prove that the two definitions of a front that we gave actually coincide.

\begin{lemma}
The explicit definition and the inductive definition of a front coincide.
\end{lemma}
\begin{proof}
\begin{descProof}
\item[Inductive $\to$ Explicit] The family $\{\emptyset\}$ is the trivial \front{}. Now let $X\in\infSub{\omega}$ and suppose that $F_{n}$ is a \front{} on $X/n$ for all $n\in X$. We need to see that $F=\Seq_{n\in X}F_{n}$ is a \front{} on $X$. Clearly $\base F=X$. If $s,t\in F$ and $s\segm t$, then for some $n\in X$ we have $\min s=\min t=n$. So for $s'=s\setminus \{n\}$ and $t'=t\setminus\{n\}$ we have $s',t'\in F_{n}$ and $s'\segm t'$, hence $s'=t'$ holds and so does $s=t$. Finally, if $Y\in\infSub{X}$ with $n=\min Y$, then there exists $s'\in F_{n}$ with $s'\segs Y\setminus\{n\}$ and so $s=\{n\}\cup s'\in F$ and $s\segs Y$. So $F$ is a \front{}, as desired.

\item[Explicit $\to$ Inductive] We show that every \front{} $F$ satisfies the inductive definition of a front by induction on the rank of $F$. If $\rk F=0$, then $F=\{\emptyset\}$ is a front according to the inductive definition. Now suppose $F$ is a front according to the explicit definition with $\rk F>0$. In particular $\base F=X$ for some $X\in \infSub{\omega}$. 
By \cref{lemFrontRays}, $F_n$ is a front on $X/n$ for every $n\in X$. Now for every $n\in X$, as $\rk F_n<\rk F$ we get that $F_n$ is a front on $X/n$ according to the inductive definition, by the induction hypothesis.
Finally as $F=\Seq_{n\in X} F_{n}$, we get that $F$ is a front according to the inductive definition. \qedhere
\end{descProof}
\end{proof}

Finally notice that the rank of a front naturally arise from the inductive definition. Let $\mathfrak{F}_{0}$ be the set containing only the trivial front. Then for any countable ordinal $\alpha$, let $F\in\mathfrak{F}_{\alpha}$ if $F\in \bigcup_{\beta<\alpha}\mathfrak{F}_{\beta}$ or $F=\Seq_{n\in X} F_{n}$ where $X\in\infSub{\omega}$ and each $F_{n}$ is a front on $X/n$ which belongs to some $\mathfrak{F}_{\beta_{n}}$ for some $\beta_{n}<\alpha$. Then clearly the set of all fronts is equal $\bigcup_{\alpha<\omega_{1}}\mathfrak{F}_{\alpha}$. Now it should be clear that for every front $F$ the smallest $\alpha<\omega_{1}$ for which $F\in \mathfrak{F}_{\alpha}$ is $\rk F$, the rank of $F$. 

\subsubsection{Sub-front and sub-super-sequences}

When using super-sequences one is often interested in extracting sub-super-sequences which enjoy further properties.

\begin{definition}
A sub-super-sequence of a super-sequence $f:F\to E$ is a restriction $f\restr{G}:G\to E$ to some front $G$ included in $F$.
\end{definition}

The following important operation allows us to understand the \emph{sub-fronts} of a given front, i.e. sub-families of a front which are themselves fronts. For a family $F\subseteq \mathcal{P}(\omega)$ and some $X\in\infSub{\omega}$, we define the sub-family
\[
F|X:=\{s\in F\mid s\subseteq X\}.
\]

\begin{proposition}
Let $F$ be a front on $X$. Then a family $F'\subseteq F$ is a front if and only if there exists $Y\in\infSub{X}$ such that $F|Y=F'$.
\end{proposition}
\begin{proof}
The claim is obvious if $F$ is trivial so suppose $F$ is non-trivial.
\begin{itemize}
\item[$\to$] Let $F'\subseteq F$ be a front on $Y$. Since $F'$ is not trivial either, $Y=\base F'\subseteq \base F=X$. Now if $s\in F'$ then clearly $s\in F|Y$. Conversely if $s\in F|Y$ then there exists a unique $t\in F'$ with $t\segs s \cup Y/s$ and so either $s\segm t$ or $t\segm s$. Since $F$ is a front and $s,t\in F$, necessarily $s=t$ and so $s\in F'$. Therefore $F'=F|Y$.
\item[$\leftarrow$] If $Y\in\infSub{X}$ then the family $F|Y$ is a front on $Y$. Clearly $F|Y$ satisfies \cref{def Front2}. If $Z\in \infSub{Y}$ then since $Y\subseteq X$, then $Z\in\infSub{X}$ and so there exists $s\in F$ with $s\segs Z$. But then $s\subseteq Z\subseteq Y$, so in fact $s\in F|Y$ and therefore $F|Y$ satisfies \cref{def Front3}. For \cref{def Front1}, notice that $\base F|Y\subseteq Y$ by definition and that if $n\in Y$, then as we have already seen there exists $s\in F|Y$ with $s\segs \{n\}\cup Y/n$, so $n\in s$ and $n\in \base F|Y$.   \qedhere
\end{itemize}
\end{proof}

Observe that the operation of restriction commutes with the taking of rays.

\begin{fact}
Let $F\subseteq \mathcal{P}(\omega)$ and $X\in\infSub{\omega}$. For every $n\in X$ we have \[F_{n}|X=(F|X)_{n}.\]
\end{fact}

Notice also the following simple important fact. If $F'$ is a sub-front of a front $F$, then the tree $T(F')$ is included in the tree $T(F)$ and so $\rk F'\leq \rk F$.

The importance of fronts essentially stems from the following fundamental theorem by Nash-Williams: Any time we partition a front into finitely many pieces, at least one of the pieces must contain a front.

\begin{theorem}[Nash-Williams]\label{cor:NashWill}
Let $F$ be a front. For any subset $S$ of $F$ there exists a front $F'\subseteq F$  such that either $F'\subseteq S$ or $F'\cap S=\emptyset$.
\end{theorem}

We now prove this theorem to give a simple example of a proof by induction on the rank of a front, a technique which is extremely fruitful.

\begin{proof}
The claim is obvious for the trivial front whose only subsets are the empty set and the whole trivial front. So suppose that the claim holds for every front of rank smaller than $\alpha$. Let $F$ be a front on $X$ with $\rk F=\alpha$ and $S\subseteq F$. For every $n\in X$ let $S_{n}$ be the subset of the ray $F_{n}$ given by $S_{n}=\{s\in F_{n}\mid \{n\} \cup s\in S\}$. 

Set $X_{-1}=X$ and $n_{0}=\min X_{-1}$. Since $\rk F_{n_{0}}<\alpha$ there exists by induction hypothesis some $X_{0}\in\infSub{X_{-1}/n_{0}}$ such that 
\[
\text{either }F_{n_{0}}|X_{0}\subseteq S_{n_{0}}, \quad\text{or }F_{n_{0}}|X_{0}\cap S_{n_{0}}=\emptyset.
\] 
Set $n_{1}=\min X_{0}$. Now applying the induction hypothesis to $F_{n_{1}}|(X_{0}/n_{1})$ and $S_{n_{1}}$ we get an $X_{1}\in \infSub{X_{0}/n_{1}}$ such that either $F_{n_{1}}|X_{1}\subseteq S_{n_{1}}$, or $F_{n_{1}}|X_{1}\cap S_{n_{1}}=\emptyset$. Continuing in this fashion, we obtain a sequence $X_{k}$ together with $n_{k}=\min X_{k-1}$ such that for all $k$ we have $X_{k}\in \infSub{X_{k-1}/n_{k}}$ and
\[
\text{either }F_{n_{k}}|X_{k}\subseteq S_{n_{k}}, \quad\text{or }F_{n_{k}}|X_{k}\cap S_{n_{k}}=\emptyset.
\]
Now there exists $Y\in \infSub{\omega}$ such that either $F_{n_{k}}|X_{k}\subseteq S_{n_{k}}$ for all $k\in Y$, or $F_{n_{k}}|X_{k}\cap S_{n_{k}}=\emptyset$ for all $k\in Y$. Let $X=\{n_{k}\mid k\in Y\}$. Then $F|X$ is as desired. Indeed for all $s\in F|X$ we have $\min s=n_{k}$ for some $k\in Y$ and $s\setminus \{n_{k}\}\in F_{n_{k}}|X_{k}$. Hence by the choice of $Y$, either $s\setminus\{\min s\}\in S_{\min s}$ for all $s\in F|X$, or $s\setminus \{\min s\}\notin S_{\min s}$ for all $s\in F|X$.  Therefore either $F|X\subseteq S$ or $F|X\cap S=\emptyset$.
\end{proof}

Nash-Williams' \cref{cor:NashWill} is easily seen to be equivalent to the following statement. 

\begin{theorem}\label{corollary:SuperNW}
Let $E$ be a finite set. Then every super-sequence $f:F\to E$ admits a constant sub-super-sequence.
\end{theorem}

The above result obviously does not hold in general for an infinite set $E$ (consider for example any injective super-sequence). However \textcite{pudlak1982partition} proved an interesting theorem in this context. In a different direction, the author proved with Carroy in \cite{CarroyYPFromWell} the following result where fronts are viewed as metric subspaces of the Cantor space $2^\omega$ by identifying subsets of $\omega$ with their characteristic functions; every super-sequence $f:F\to E$ in some compact metric space $E$ admits a sub-super-sequence which is \emph{uniformly continuous}.

\subsection{Multi-sequences}\label{sec:MultisSeqSimpson}

Another approach to super-sequences initiated by \textcite{simpson1985bqo} has proved very useful in the theory of better-quasi-orders. We now describe this approach and relate it to super-sequences.

Let $E$ be any set, and $f:F\to E$ be a super-sequence with $F$ a front on $X$. By the explicit definition of front for every $Y\in\infSub{X}$ there exists a unique $s\in F$ with $s\segs Y$. We can therefore define a map $f^{\upcl}:\infSub{X}\to E$ defined by $f^{\upcl}(Y)=f(s)$ where $s$ is the unique member of $F$ with $s\segs Y$.

\begin{definition}
A \emph{multi-sequence} into some set $E$ is a map $h:\infSub{X}\to E$ for some $X\in\infSub{\omega}$. A \emph{sub-multi-sequence} of $h:\infSub{X}\to E$ is a restriction of $h$ to $\infSub{Y}$ for some $Y\in\infSub{X}$.
\end{definition}

For every $X\in\infSub{\omega}$ we endow $\infSub{X}$ with the topology induced by the Cantor space, viewing subsets as their characteristic functions. As a topological space $\infSub{X}$ is homeomorphic to the Baire space $\bai$. This homeomorphism is conveniently realised via the embedding of $\infSub{X}$ into $\bai$ which maps each $Y\in \infSub{X}$ to its injective and increasing enumeration $e_{Y}:\omega \to Y$. We henceforth identify the space $\infSub{X}$ with the closed subset of $\bai$ of injective and increasing sequences in $X$. From this point of view we have a countable basis of clopen sets for $\infSub{X}$ consisting in sets of the form
\[
M_{s}=\{Y\in\infSub{X}\mid s\segs Y\}, \quad \text{for $s\in \finSub{X}$.}
\]

\begin{definition}
A multi-sequence $h:\infSub{X}\to E$ is \emph{locally constant} if for all $Y\in\infSub{X}$ there exists $s\in\finSub{X}$ such that $Y\in M_{s}$ and $h$ is constant on $M_{s}$, i.e. for every $Y\in\infSub{X}$ there exists $s\segs Y$ such that for every $Z\in\infSub{X}$, $s\segs Z$ implies $h(Z)=h(Y)$.
\end{definition}

Clearly for every super-sequence $f:F\to E$ where $F$ is a front on $X$ the map $f^{\upcl}:\infSub{X}\to E$ is locally constant. 

Conversely for any locally constant multi-sequence $h:\infSub{X}\to E$, let 
 \[
 S^{h}=\{s\in \finSub{X}\mid \text{$h$ is constant on $M_{s}$}\}.
 \]
 
 \begin{lemma}
 The set $F^{h}$ of $\segm$-minimal elements of $S^{h}$ is a front on $X$.
 \end{lemma}
 \begin{proof}
 By $\segm$-minimality if $s,t\in F^{h}$ and $s\segm t$, then $s=t$. For every $Y\in\infSub{X}$, since $h$ is locally constant there exists $s\segs Y$ such that $h$ is constant on $M_{s}$. Hence there exists $t\in F^{h}$ with $t\segm s$, and so $t\segs Y$ too. To see that either $F^{h}$ is trivial or $\base F^{h}=X$, notice that $h$ is constant if and only if $F^{h}$ is the trivial front if and only if $\emptyset\in F_{h}$. So if $F^{h}$ is not trivial, then for every $n\in X$ there exists $s\in F^{h}$ with $s\segs \{n\}\cup X/n$ and since $s\neq\emptyset$, we get $n\in s$ and $n\in \base F^{h}$.  
 \end{proof}

We can therefore associate to every locally constant multi-sequence $h:\infSub{X}\to E$ a super-sequence $h^{\dcl}:F^{h}\to E$ by letting, in the obvious way, $h^{\dcl}(s)$ be equal to the unique value taken by $h$ on $M_{s}$ for every $s\in F^{h}$.

\begin{remark}
Clearly every front arises as an $F^{h}$ for some locally constant multi-sequence $h$. Indeed for any front $F$ and any injective super-sequence $f$ from $F$, we have $F=F^{f^{\upcl}}$. Therefore we can think of the definition of a front as a characterisation of those families of finite subsets of $\omega$ arising as an $F^{h}$ for some locally constant multi-sequence $h$.
\end{remark}

The basic properties of the correspondence $h\mapsto h^{\dcl}$ and $f\mapsto f^{\upcl}$ are easily stated with the help of the following partial order among super-sequences in a given set.
 
 \begin{definition}
Let both $F$ and $G$ be fronts on the same set $X\in\infSub{\omega}$ and $f:F\to E$ and $g:G\to E$ be any maps. We write $f\segm g$ when
\begin{enumDefn}
\item for every $s\in F$ there exists $t\in G$ with $s\segm t$, and
\item for every $s\in F$ and every $t\in G$, $s\segm t$ implies $f(s)=g(t)$.
\end{enumDefn} 
\end{definition}

To simplify notation we write $\check{f}:\check{F}\to E$ instead of $(f^{\upcl})^{\dcl}:F^{f^{\upcl}}\to E$.

 \begin{fact}\label{fac SpareIntOperator}
 Let $X\in\infSub{\omega}$ and $E$ be a set.
 \begin{enumThm}
 \item for every front $F$ on $X$ and every map $f:F\to E$, the map $\check{f}:\check{F}\to E$ is such that $\check{f}\segm f$.

 \item for every fronts $F$ and $G$ on $X$ and maps $f:F\to E$ and $g:G\to E$, $f\segm g$ implies $f^{\upcl}=g^{\upcl}$.
 \item for every locally constant map $h:\infSub{X}\to E$, we have $(h^{\dcl})^{\upcl}=h$.
 \end{enumThm}
 \end{fact}
 
 It follows that for every locally constant multi-sequence $h:\infSub{X}\to E$ the super-sequence $h^{\dcl}:F^{h}\to E$ is the minimal element for $\segm$ among the set of super-sequences $g:G\to E$ with $g^{\upcl}=h$. Moreover for every super-sequence $f:F\to E$ the super-sequence $\check{f}:\check{F}\to E$ is the $\segm$-minimal among the super-sequences $g$ with $g\segm f$. In particular $\check{\check{f}}=\check{f}$ for every super-sequence $f$.
 
The super-sequences which are $\segm$-minimal sometimes play a role and we now give them a name.
 
 \begin{definition}
 Let $E$ be a set and $F$ a front on $X$. A super-sequence $f:F\to E$ is said to be \emph{spare} if $f$ is minimal for $\segm$, or equivalently $\check{f}=f$, i.e. if $\check{F}=F$.
 \end{definition}

\begin{example}
If $F$ is a non trivial front and $c:F\to E$ is constant equal to $e\in E$ then $c$ is not spare and of course $\check{c}:\{\emptyset\}\to E$, $\emptyset\mapsto e$.
\end{example}

The following is a simple characterisation of spare super-sequences.

\begin{lemma} \label{lem Spare}
Let $f:F\to E$ be a super-sequence in some set $E$. Then the following are equivalent
\begin{enumThm}
\item $f$ is spare,
\item Whenever $t\in F$ and $s\segs t$, then there exists $t'\in F$ with $s\segs t'$ and $f(t)\neq f(t')$.
\end{enumThm} 
\end{lemma}
\begin{proof}
Suppose that $s\segs t \in F$ and $f(t)=f(t')$ for every $t'\in F$ with $s\segm t'$. It follows that $f^{\upcl}$ is constant on $M_{s}$ but so $t\notin \check{F}$ and therefore $f$ is not spare.

Conversely if $f$ is not spare, then there exists some $t\in F$ which is not $\segm$-minimal in $F^{f^{\upcl}}$. This means that there is $s\in \check{F}$ with $s\segs t$ and $f^{\upcl}$ is constant on $M_{s}$, so for every $t'\in F$ with $s\segm t'$ we have $f(t)=f(t')$.
\end{proof}

\subsection{Iterated powerset, determinacy of finite games}\label{sec bqoPower}
 
  \begin{quote}
{\itshape It also transpires that if, by a certain fairly natural extension of our definition of \textnormal{[$\mathcal{P}^{n}(Q)$]}, we define \textnormal{[$\mathcal{P}^{\alpha}(Q)$]} for every ordinal $\alpha$, then $Q$ is bqo iff \textnormal{[$\mathcal{P}^{\alpha}(Q)$]} is wqo for every ordinal $\alpha$. To justify these statements would not be relevant here, but it was from this point of view that the author was first led to study bqo sets.}
 
 \flushright Crispin St. John Alvah Nash-Williams \cite[p. 700]{nash1965welltrees}
 \end{quote}

\medskip
Following in Nash-Williams' steps, we introduce the notion of better-quasi-orders as the quasi-orders whose iterated powersets are \wqo{}. We do this in the light of further developments of the theory, taking advantage of Simpson's point of view on super-sequences, using the determinacy of finite games and a powerful game-theoretic technique invented by Tony Martin.     

First let us define precisely the iterated powerset of a qo together with its lifted quasi-order. To facilitate the following discussion we focus on the non-empty sets over some quasi-order $Q$.  Let $\mathcal{P}^{*}(A)$ denote the set of non-empty subsets of a set $A$, i.e. $\mathcal{P}^{*}(A)=\mathcal{P}(A)\setminus \{\emptyset\}$. We define by transfinite recursion 
\begin{align*}
\PowerQ_{0}(Q)&=Q\\
\PowerQ_{\alpha+1}(Q)&=\mathcal{P}^{*}(\PowerQ_{\alpha}(Q))\\
\PowerQ_{\lambda}(Q)&=\bigcup_{\alpha<\lambda}\PowerQ_{\alpha}(Q), \quad \text{for $\lambda$ limit.}
\end{align*}
We treat the element of $Q$ as \emph{urelements} or \emph{atoms}, namely they have no elements but they are different from the empty set. 
Let
\[
\PowerQ(Q)=\bigcup_{\alpha} \PowerQ_{\alpha}(Q).
\]

Let us define the \emph{support} of $X\in\PowerQ(Q)$, denoted by $\supp(X)$, by induction on the membership relation as follows: if $q\in Q$, then $\supp(q)=\{q\}$, otherwise let $\supp(X)=\bigcup \{ \supp(x)\mid x\in X\}$. Notice that for every subset $X$ of $Q$ we actually have $\supp(X)=X$.


Following an idea of \textcite{forster2003better} we define the quasi-order on $\PowerQ(Q)$ via the existence of a winning strategy in a natural game. We refer the reader to \textcite[(20.)]{kechris1995classical} for the basic definitions pertaining to two-player games with perfect information.

\begin{definition}
For every $X,Y\in\PowerQ(Q)$ we define a two-player game with perfect information $\GPowerQ (X,Y)$ by induction on the membership relation. The game $\GPowerQ (X,Y)$ goes as follows. Player $\I$ starts by choosing some $X'$ such that:
\begin{itemize}
\item if $X\notin Q$, then $X'\in X$,
\item otherwise, $X'=X$.
\end{itemize}
Then Player $\II$ replies by choosing some $Y'$ such that:
\begin{itemize}
\item if $Y\notin Q$, then $Y'\in Y$,
\item otherwise $Y'=Y$.
\end{itemize}
If both $X'$ and $Y'$ belong to $Q$, then Player $\II$ wins if $X'\leq Y'$ in $Q$ and Player $\I$ wins if $X'\nleq Y'$ . Otherwise the game continues as in $\GPowerQ (X',Y')$. 
\end{definition}

We then define the lifted quasi-order on $\PowerQ(Q)$ by letting for $X,Y\in\PowerQ(Q)$
\[
X\leq Y \SSI \text{Player $\II$ has a winning strategy in $\GPowerQ (X,Y)$.}
\]

\begin{remark}\label{rem PowerDownset}
The above definition can be rephrased by induction on the membership relation as follows:
\begin{enumDefn}
\item if $X,Y\in Q$, then $X\leq Y$ if and only if $X\leq Y$ in $Q$,
\item if $X\in Q$ and $Y\notin Q$, then 
\[
X\leq Y \SSI \text{there exists $Y'\in Y$ with $X\leq Y'$},
\]
\item if $X\notin Q$ and $Y\in Q$, then \label{def QoPowerQ}
\[
X\leq Y \SSI \text{for every $X'\in X$ we have $X'\leq Y$},
\]
\item if $X\notin Q$ and $Y\notin Q$, then 
\[
X\leq Y \SSI \text{for every $X'\in X$ there exists $Y'\in Y$ with $X'\leq Y'$}.
\]
\end{enumDefn}
Our definition coincides with the one given by \textcite[Claim 1.7, p.188]{shelah1982better}. But  \textcites{milner1985basic}{laverfraisse} both omit condition \cref{def QoPowerQ}. 
\end{remark}

The axiom of foundation ensures that in any play of a game $\GPowerQ (X,Y)$ a round where both players have chosen elements of $Q$ is eventually reached, resulting in the victory of one of the two players. In particular, each game $\GPowerQ (X,Y)$ is determined as already proved by \textcite{von1944theory} (see \cite[(20.1)]{kechris1995classical}).
The crucial advantage of the game-theoretic formulation of the quasi-order on $\PowerQ(Q)$ resides in the fact that the negative condition $X\nleq Y$ is equivalent to the existential statement \enquote{Player $\I$ has a winning strategy}. 

Now suppose $Q$ is a quasi-order such that $\PowerQ(Q)$ is not \wqo{} and let $(X_{n})_{n\in\omega}$ be a bad sequence in $\PowerQ(Q)$. Whenever $m<n$ we have $X_{m}\nleq X_{n}$ and we can choose a winning strategy $\sigma_{m,n}$ for Player $\I$ in $\GPowerQ (X_{m},X_{n})$. We define a locally constant multi-sequence $g:\infSub{\omega}\to Q$ as follows. Let $N=\{n_{0},n_{1},n_{2},\ldots \}$ be an infinite subset of $\omega$ enumerated in increasing order. We define $g(N)$ as the last move of Player $\I$ in a particular play of $\GPowerQ (X_{n_{0}},X_{n_{1}})$ in a way best understood by contemplating \cref{fig ConstrMultiSeq}. 
\begin{figure}
\begin{center}
\includegraphics{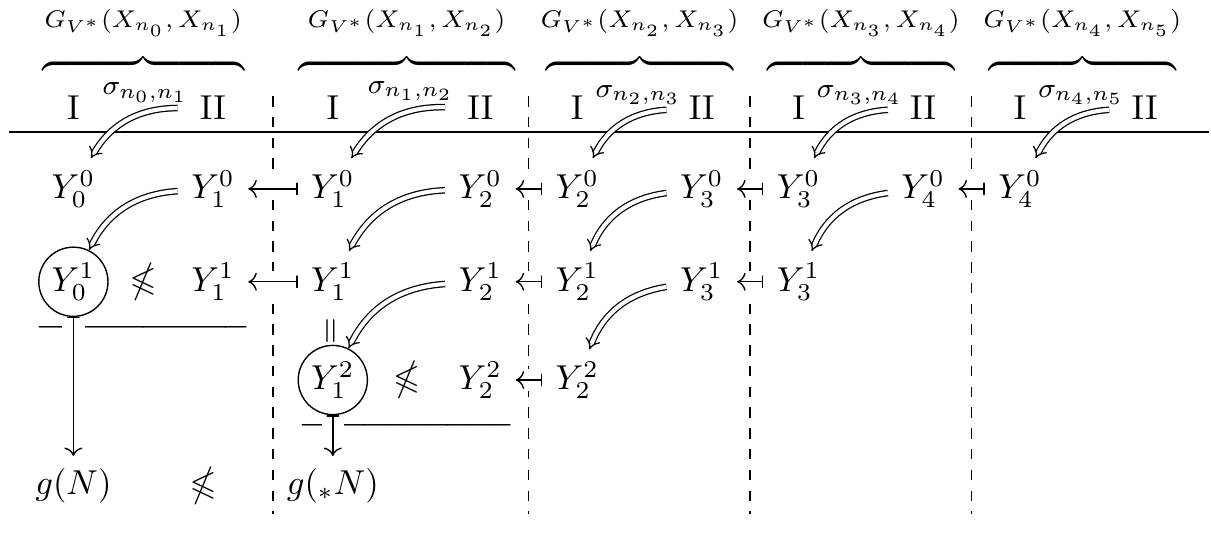} 
\end{center}
\caption{Constructing a multi-sequence by stringing strategies together.}
\label{fig ConstrMultiSeq}
\end{figure}

Let $Y^{0}_{0}$ be the the first move of Player $\I$ in $\GPowerQ (X_{n_{0}},X_{n_{1}})$ as prescribed by its winning strategy $\sigma_{n_{0},n_{1}}$. Then let Player $\II$ copy the first move $Y^{0}_{1}$ of Player $\I$ given by the strategy $ \sigma_{n_{1},n_{2}}$ in $\GPowerQ (X_{n_{1}},X_{n_{2}})$. Then Player $\I$ answers $Y^{1}_{0}$ according to the strategy $\sigma_{n_{0},n_{1}}$. Now if $Y^{0}_{1}$ is not in $Q$, then we need to continue our play of $\GPowerQ (X_{n_{1}},X_{n_{2}})$ a little further to determine the second move of Player $\II$ in $\GPowerQ (X_{n_{0}},X_{n_{1}})$. Let the first move of Player $\II$ in $\GPowerQ (X_{n_{1}},X_{n_{2}})$ be the first move of Player $\I$ in $\GPowerQ (X_{n_{2}},X_{n_{3}})$ as prescribed by his winning strategy $\sigma_{n_{2},n_{3}}$. Then this determines the second move $Y^{1}_{1}$ of Player $\I$ in $\GPowerQ (X_{n_{1}},X_{n_{2}})$ according to $\sigma_{n_{1},n_{2}}$. We then let the second move of Player $\II$ in $\GPowerQ (X_{n_{0}},X_{n_{1}})$ to be this $Y^{1}_{1}$. This yields some answer $Y^{1}_{0}$ of Player $\I$ according to $\sigma_{n_{0},n_{1}}$. We continue so on and so forth until the play of $\GPowerQ (X_{n_{0}},X_{n_{1}})$ reaches an end with some $(Y^{k_{N}}_{0},Y^{k_{N}}_{1})\in Q\times Q$ and we let $g(N)=Y^{k_{N}}_{0}$.  Since the play of $\GPowerQ (X_{n_{0}},X_{n_{1}})$ is finite, $g(N)$ depends only on a finite initial segment of $N$ and we have therefore defined a locally constant multi-sequence $g:\infSub{\omega}\to Q$. 

Now since Player $\I$ has followed the winning strategy $\sigma_{n_{0},n_{1}}$ we have $Y^{k_{N}}_{0}\nleq Y^{k_{N}}_{1}$. Now if the play of the game $\GPowerQ (X_{n_{1}},X_{n_{2}})$ has not yet reached an end at step $k_{N}$ we go on in the same fashion. Assume it ends with some pair $(Y^{l}_{1},Y^{l}_{2})$ in $Q$. By the rules of the game $\GPowerQ $, since $Y_{1}^{k_{N}}\in Q$ we necessarily have $Y^{l}_{1}=Y^{k_{N}}_{1}$. But $Y^{l}_{1}$ is just $g(\{n_{1},n_{2},n_{3},\ldots\})$, hence for every $N\in\infSub{\omega}$ we have
\[
g(N)\nleq g\big(N\setminus \{\min N\}\big).
\]

For every $N\in\infSub{\omega}$ we call the \emph{shift of $N$}, denoted by $\shift{N}$, the set $N\setminus\{\min N\}$. We are led to the following:
\begin{definition}
Let $Q$ be a qo and $h:\infSub{X}\to Q$ a multi-sequence.
\begin{enumDefn}
\item We say that $h$ is \emph{bad} if $h(N)\nleq h(\shift{N})$ for every $N\in\infSub{X}$,
\item We say that $h$ is \emph{good} if there exists $N\in\infSub{X}$ with $h(N)\leq h(\shift{N})$,
\end{enumDefn}
\end{definition}

At last, we present the deep definition due to Nash-Williams here in a modern reformulation.

 \begin{definition}
A quasi-order $Q$ is a \emph{better-quasi-order} (\bqo{}) if there is no bad locally constant multi-sequence in $Q$.
 \end{definition}

Of course the definition of better-quasi-order can be formulated in terms of super-sequences as Nash-Williams originally did.
The only missing ingredient is a counterpart of the shift map $N\mapsto \shift{N}$ on finite subsets of natural numbers.

\begin{definition}
For $s,t\in\finSub{\omega}$ we say that $t$ is a shift of $s$ and write $s\tri t$ if there exists $X\in\infSub{\omega}$ such that 
\[s\segs X\mbox{ and }t\segs \shift{X}.\]
\end{definition}

\begin{definitions} 
Let $Q$ be a qo and $f:F\to Q$ be a super-sequence.
\begin{enumDefn}
\item We say that $f$ is \emph{bad} if whenever $s\tri t$ in $F$, we have $f(s)\nleq f(t)$.
\item We say that $f$ is \emph{good} if there exists $s,t\in F$ with $s\tri t$ and $f(s)\leq f(t)$.
\end{enumDefn}
\end{definitions}

 \begin{lemma}\label{badSimpsonNW}
 Let $Q$ be a quasi-order. 
 \begin{enumThm}
 \item If $h:\infSub{\omega}\to Q$ is locally constant and bad, then $h^{\dcl}:F^{h}\to Q$ is a bad super-sequence. \label{badSimpsonNW1}
 \item If $f:F\to Q$ is a bad super-sequence from a front on $X$, then $f^{\upcl}:\infSub{X}\to Q$ is a bad locally constant multi-sequence.\label{badSimpsonNW2}
 \end{enumThm}
 \end{lemma}
\begin{proof}
\begin{itemize} 
\item[\cref{badSimpsonNW1}] Suppose $h:\infSub{X}\to Q$ is locally constant and bad. Let us show that $h^{\dcl}:F^{h}\to Q$ is bad. If $s,t\in F^{h}$ with $s\tri t$, i.e. there exists $Y\in\infSub{X}$ such that $s\segs Y$ and $t\segs \shift{Y}$. Then $h^{\dcl}(s)=h(Y)$ and $h^{\dcl}(t)=h(\shift{Y})$ and since $h$ is assumed to be bad, we have $h^{\dcl}(s)\nleq h^{\dcl}(t)$.
 
\item[\cref{badSimpsonNW2}] Suppose $f:F\to Q$ is bad from a front on $X$ and let $Y\in\infSub{X}$. There are unique $s,t\in F$ such that $s\segs Y$ and $t\segs \shift{Y}$, and clearly $f^{\upcl}(Y)=f(s)$, $f^{\upcl}(\shift{Y})=f(t)$, and $s\tri t$. Therefore $f^{\upcl}(Y)\nleq f^{\upcl}(\shift{Y})$ holds. \qedhere
\end{itemize}
 \end{proof}

 \begin{proposition}\label{cor WlogSpare}
For a quasi-order $Q$ the following are equivalent.
\begin{enumThm}
\item $Q$ is a better-quasi-order,
\item there is no bad super-sequence in $Q$,
\item there is no bad spare super-sequence in $Q$.
\end{enumThm}
 \end{proposition}

 The idea of stringing strategies together that we used to arrive at the definition of \bqo{} is directly inspired from a famous technique used by \textcite[Theorem 3.2]{van1987rigid} together with \textcite[Theorem 3]{louveau1990quasi}. This method was first applied by Martin in the proof of the well-foundedness of the Wadge hierarchy (see \cite[(21.15), p. 158]{kechris1995classical}). \Textcite{forster2003better} introduces better-quasi-orders in a very similar way, but a super-sequence instead of a multi-sequence is constructed, making the similarity with the method used by \textcite{van1987rigid,louveau1990quasi} less obvious. One of the advantages of multi-sequences resides in the fact that they enable us to work with super-sequences without explicitly referring to their domains. This is particularly useful in the above construction, since a bad sequence in $\PowerQ(Q)$ can yield a multi-sequence whose underlying front is of arbitrarily large rank.  Indeed \textcite{marcone1994foundations} showed that super-sequences from fronts of arbitrarily large rank are required in the definition of \bqo{}.

Notice that the notion of \bqo{} naturally lies between those of well-orders and \wqo.

\begin{proposition}\label{prop WellOrderWqoBqo}
Let $Q$ be a qo. Then
\[
\text{$Q$ is a well-order} \quad\to\quad \text{$Q$ is \bqo{}}\quad \to\quad \text{$Q$ is \wqo{}.}
\]
\end{proposition}
\begin{proof}
Suppose $Q$ is a well order and let $h:\infSub{\omega}\to Q$ is any multi-sequence in $Q$. Fix $X\in\infSub{\omega}$ and let $X_{0}=X$ and $X_{n+1}=\shift{X_{n}}$. Since $Q$ is a well-order, there exists $n$ such that $h(X_{n})\leq h(X_{n+1})$, otherwise $h(X_{n})$ would be a descending chain in $Q$. So $h$ is good and therefore $Q$ is \bqo{}.

Now observe that for $m,n\in\{\omega\}$ we have $\{m\}\tri \{n\}$ if and only if $m<n$. So if $Q$ is \bqo{}, then in particular every sequence $f:[\omega]^{1}\to Q$ is good, and so $Q$ is \wqo{}.
\end{proof}

\subsection{Equivalence}\label{subsec BqoEquiv}

Pushing further the idea that led us to the definition of \bqo{}, we can build from any bad multi-sequence in $\PowerQ(Q)$ a bad multi-sequence in $Q$. Therefore proving that if $Q$ is \bqo{}, then $\PowerQ(Q)$ is actually \bqo{}.

\begin{proposition}\label{prop ReflectLocConst}
Let $Q$ be a qo. For every bad locally constant $h:\infSub{\omega}\to \PowerQ(Q)$ there exists a bad locally constant $g:\infSub{\omega}\to Q$ such that moreover $g(X)\in \supp(h(X))$ for every $X\in\infSub{\omega}$.
\end{proposition}

\begin{proof}
Let $h:\infSub{\omega}\to \PowerQ(Q)$ be locally constant and bad, and let us write $h(X)=h_{X}$ for $X\in\infSub{\omega}$. Notice that the image of $h$ is countable and choose for every $X\in\infSub{\omega}$ a winning strategy $\sigma_{X}$ for Player $\I$ in $\GPowerQ (h_{X},h_{\shift{X}})$. We let $X_{0}=X$ and $X_{n+1}=\shift{X_{n}}$. 
\begin{figure}[ht]
\begin{center}
\includegraphics{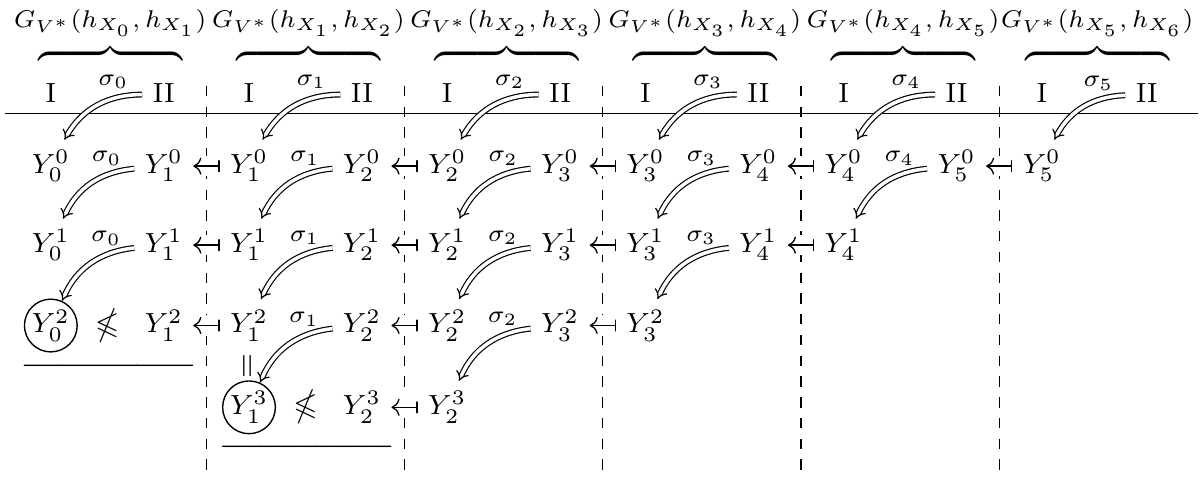}
\end{center}
\caption{Stringing strategies together.}
\label{fig PlayingCopying}
\end{figure}

Consider the diagram in \cref{fig PlayingCopying} obtained by letting Player $\I$ follow the winning strategy $\sigma_{n}=\sigma_{X_{n}}$ in $\GPowerQ (h_{X_{n}},h_{X_{n+1}})$ and $\II$ responding in $\GPowerQ (h_{X_{n}},h_{X_{n+1}})$ by copying $\I$'s moves in $\GPowerQ (h_{X_{n+1}},h_{X_{n+2}})$. This uniquely determines for each $n$ a finite play $(Y^{i}_{n},Y^{i}_{n+1})_{i \leq l_{n}}$ of the game $\GPowerQ (h_{X_{n}},h_{X_{n+1}})$ ending with some $Y^{l_{n}}_{n}\nleq Y^{l_{n}}_{n+1}$ in $Q$.  Clearly the play $(Y^{i}_{n},Y^{i}_{n+1})_{i \leq l_{n}}$ depends only on the value taken by $h$ on the $X_{j}$ with  $j\in\{n,\ldots ,n+l_{n}+2\}$. By the rules of the game $\GPowerQ$ for every $n$ we have $Y^{l_{n}}_{n+1}=Y^{l_{n+1}}_{n+1}$. We let $Y^{X}_{0}=Y^{l_{0}}_{0}$ and $Y^{X}_{n+1}=Y^{l_{n}}_{n+1}=Y^{l_{n+1}}_{n+1}$. We define $g:\infSub{\omega}\to Q$ by letting $g(X)=Y^{X}_{0}$. Since $Y^{X}_{0}$ depends only on $h_{X_{0}},\ldots h_{X_{l_{0}+2}}$ and $h$ is locally constant, it follows that $g$ is locally constant. Moreover, by construction $g(\shift{X})=Y^{\shift{X}}_{0}=Y^{X}_{1}$ and so $g(X)\nleq g(\shift{X})$.
\end{proof}

\begin{corollary}\label{cor PowerQPresBqo}
If $Q$ is \bqo{}, then $\PowerQ(Q)$ is \bqo{}.
\end{corollary}

We now briefly show that there is a strong converse to \cref{cor PowerQPresBqo}.

Let $f:F\to Q$ be a super-sequence from a front on $\omega$ in a qo $Q$. Remember from \cref{lem TreeFrontWellFounded}, that the tree $T(F)=\{s\in \finSub{\omega}\mid \exists t\in F \ s\segm t\}$ is well-founded. We define by recursion on the well-founded relation $\sges$ on $T(F)$ a map $\tilde{f}:T(F)\to \PowerQ(Q)$ by
\begin{align*}
\tilde{f}(s)&=f(s) &&\text{if $s\in F$,}\\
\tilde{f}(s)&=\big\{\tilde{f}(s\cup \{n\})\mid n\in\omega/s \text{ and } s\cup\{n\}\in  T(F)\big\} &&\text{otherwise.}
\end{align*}
As long as $F$ is not trivial we have $[\omega]^{1}\subseteq T(F)$ and restricting $\tilde{f}$ to $[\omega]^{1}$ we obtain the sequence $\tilde{f}\restr{[\omega]^{1}}:[\omega]^{1}\to \PowerQ(Q)$. Notice also that $\tilde{f}(s)\in Q$ if and only if $s\in F$. 

\begin{lemma}\label{lem: ConverseGame}
If $f:F\to Q$ is bad, then $\widetilde{f}\restr{[\omega]^{1}}$ is a bad sequence in $\PowerQ(Q)$.
\end{lemma}
\begin{proof}
By way of contradiction suppose that for some $m_{0},n_{0}\in \omega$ with $m_{0}<n_{0}$ we have $\tilde{f}(m_{0})\leq\tilde{f}(n_{0})$ in $\PowerQ(Q)$ and let $\sigma$ be a winning strategy for Player $\II$ in $\GPowerQ \big(\tilde{f}(m_{0}),\tilde{f}(n_{0})\big)$. Let $s_{0}=(m_{0})$, $t_{0}=(n_{0})$ and $u_{0}=(m_{0},n_{0})$. We consider the following play of $\GPowerQ \big(\tilde{f}(m_{0}),\tilde{f}(n_{0})\big)$. Observe that if $s_{0}=(m_{0})\notin F$, then $u_{0}=(m_{0},n_{0})\in T(F)$. We make Player $\I$ start with $\tilde{f}(s_{1})$ where $s_{1}=s_{0}$ if $s_{0}\in F$ and $s_{1}=u_{0}$ otherwise. Then $\II$ answers according to $\sigma$ by $\tilde{f}(t_{1})$ for some $t_{1}\in T(F)$. If $t_{0}=(n_{0})\in F$, then necessarily $t_{1}=t_{0}$ and we let $u_{1}=u_{0}\conc (k)$ with $k=1+\max u_{0}$. Otherwise $t_{0}\segs t_{1}$ and $t_{1}=(n_{0},n_{1})$ for some $n_{1}>n_{0}$, we then let $u_{1}=u_{0}\cup t_{1}=u_{0}\conc (n_{1})$. Notice that in any case $s_{1}\tri t_{1}$ since for $X=u_{1}\cup \omega/u_{1}$ we have $s_{1}\segs X$ and $t_{1}\segs \shift{X}$. Then we make $\I$ respond with $\tilde{f}(s_{2})$ where $s_{2}=s_{1}$ if $s_{1}\in F$, $s_{2}=u_{1}$ if $s_{1}\notin F$. We continue in this fashion, an example of which is depicted in  \cref{fig PlayingCopyingReverse}. After finitely many rounds $\I$ has reached some $f(s)$ for $s\in F$, and $\II$ has reached some $f(t)$ with $t\in F$. By construction $s\tri t$, but since $\sigma$ is winning for $\II$, we have $f(s)\leq f(t)$, a contradiction.
\begin{figure}[!ht]
\begin{center}
\includegraphics{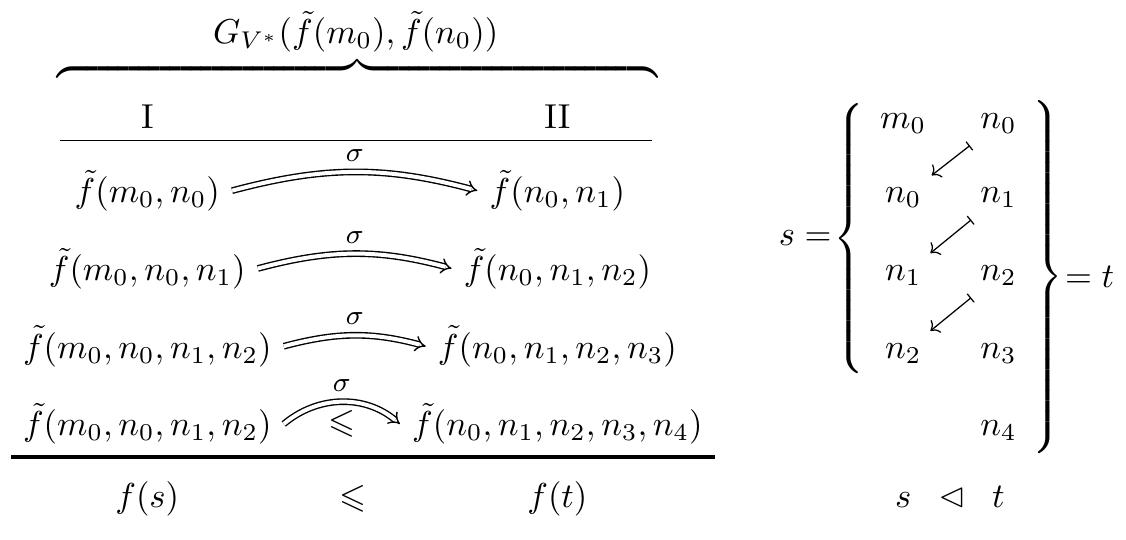}
\end{center}
\caption{Copying and shift.}
\label{fig PlayingCopyingReverse}
\end{figure}
\end{proof}

Notice that by definition $\tilde{f}:T(F)\to \PowerQ(Q)$ only reaches \emph{hereditarily countable non-empty sets over} $Q$, namely elements of $Q$ and countable non-empty sets of hereditarily countable non-empty sets over $Q$. Let $\HerCtblQ(Q)$ denote the set of hereditarily countable non-empty sets over $Q$ equipped with the qo induced from $\PowerQ(Q)$. 
We have obtained the following well known equivalence.

\begin{theorem}\label{thm BqoHerCtblWqo}
A quasi-order $Q$ is \bqo{} if and only if $\HerCtblQ(Q)$ is \wqo{}.
\end{theorem}
\begin{proof}
If $Q$ is \bqo{} then $\PowerQ(Q)$ is \bqo{} by \cref{cor PowerQPresBqo} and so in particular $\HerCtblQ(Q)$ is \wqo{}. For the converse implication, assume that $Q$ is not \bqo{}. Then there is some bad super-sequence in $Q$ and \cref{lem: ConverseGame} yields a bad sequence in $\HerCtblQ(Q)$, so $\HerCtblQ(Q)$ is not \wqo{}.
\end{proof}

Notice that by definition any countable non-empty subset of $\HerCtblQ(Q)$ belongs to $\HerCtblQ(Q)$. Moreover, by \cref{prop:ClassicEquWqo2}~\cref{Item:5b} a quasi-order is \wqo{} if and only if the qo $\CtblSubs(Q)$ of its countable subsets is well-founded, so $\HerCtblQ(Q)$ is \wqo{} if and only if it is well-founded.

\begin{theorem}\label{thm BqoHerCtblWellF}
A quasi-order $Q$ is \bqo{} if and only if $\HerCtblQ(Q)$ is well-founded.
\end{theorem}

\section{Around the definition of better-quasi-order}\label{sec AroundBQO}

In the previous section, we were led to the definition of \bqo{}s by reflecting a bad sequence in $\PowerQ(Q)$ into some bad multi-sequence in $Q$. In this section, we discuss the definition we obtained and try to understand what its essential features are. Along this line we show that the presence of the shift is somewhat accidental. 

\subsection{The perfect versus bad dichotomy}

For every $X\in\infSub{\omega}$, let us denote invariably by $\SHIFT:\infSub{X}\to\infSub{X}$ the \emph{shift map} defined by $\SHIFT(N)=\shift{N}$ for every $N\in\infSub{X}$.

For our discussion, we wish to treat both the pairs $(\infSub{X},\SHIFT)$, for $X\in\infSub{\omega}$, and the quasi-orders $(Q,\leq)$ as objects in the same category. 

With this in mind, let us call a \emph{topological digraph} a pair $(A,R)$ consisting of a topological space $A$ together with a binary relation $R$ on $A$. If $(A,R)$ and $(B,S)$ are topological digraphs, a \emph{continuous homomorphism} from $(A,R)$ to $(B,S)$ is a continuous map $\varphi:A\to B$ such that for every $a,a'\in A$, $a\mathrel{R}a'$ implies $\varphi(a)\mathrel{S}\varphi(a')$. As an important particular case, if $f:A\to A$ is any function we write $(A,f)$ for the topological digraph whose binary relation is the graph of the function $f$. If $f:A\to A$ and $g:B\to B$ are functions, a map $\varphi:A\to B$ is a continuous homomorphism from $(A,f)$ to $(B,g)$ exactly in case $\varphi$ is continuous and $\varphi\circ f=g\circ \varphi$. For a binary relation $R$ on $A$ let us denote by $R^{\complement}$ the binary relation $(A\times A)\setminus R$. 

Observe that for a discrete space $A$, a multi-sequence $h:\infSub{X}\to A$ is continuous exactly when it is locally constant.

\begin{proposition}\label{prop: DichtomyMultiseq}
Let $f:\infSub{\omega}\to\infSub{\omega}$ be a continuous map such that $f(X)\subseteq X$ for every $X\in\infSub{\omega}$ and $R$ be a binary relation on a discrete space $A$. For every continuous $\varphi:\infSub{\omega}\to A$ there exists $Z\in\infSub{\omega}$ such that
\begin{descThm}
\item[either] $\varphi: (\infSub{Z}, f)\to (A,R)$ is a continuous homomorphism,
\item[or] $\varphi:(\infSub{Z},f)\to (A,R^{\complement})$ is a continuous homomorphism.
\end{descThm}
\end{proposition}
\begin{proof}
Let $\varphi:\infSub{\omega}\to (A,R)$ be locally constant and define $c:\infSub{\omega}\to 2$ by $c(X)=1$ if and only if $\varphi(X)\mathrel{R}\varphi(f(X))$. Clearly $c$ is locally constant so let $c^{\dcl}:F^{c}\to 2$ be the associated super-sequence. By Nash-Williams' \cref{cor:NashWill} there exists an infinite subset $Z$ of $\omega$ such that $c^{\dcl}\restr{F^{c}|Z}:F^{c}|Z\to 2$ is constant. Therefore for the restriction $\psi=\varphi\restr{\infSub{Z}}:\infSub{Z}\to A$ it follows that either $\psi:(\infSub{Z},f)\to (A,R^{\complement})$ is a continuous homomorphism, or $\psi:(\infSub{Z},f)\to (A,R)$ is a continuous homomorphism.
\end{proof}

\begin{remark}
The previous proposition generalises as follows. Let $A$ be any topological space, $R\subseteq A\times A$ be a Borel binary relation and $f:\infSub{\omega}\to\infSub{\omega}$ a Borel map such that $f(X)\subseteq X$ for every $X\in\infSub{\omega}$. For every Borel map $\varphi:\infSub{\omega}\to A$ there exists $Z\in\infSub{\omega}$ such that
\begin{descThm}
\item[either] $\varphi: (\infSub{Z}, f)\to (A,R)$ is a Borel homomorphism,
\item[or] $\varphi:(\infSub{Z},f)\to (A,R^{\complement})$ is a Borel homomorphism.
\end{descThm}
Indeed, the set
\[
\{X\in\infSub{\omega}\mid \varphi(X)\mathrel{R}\varphi(f(X))\}=\big(\varphi\times(\varphi \circ f)\big)^{-1}(R)
\]
is Borel in $\infSub{\omega}$ and thus, by the Galvin-Prikry theorem \cite{JSL:9194679}, there exists a $Z\in\infSub{\omega}$ as required.
\end{remark}

\begin{definition}
Let $R$ be a binary relation on a discrete space $A$.
\begin{enumDefn}
\item A multi-sequence $h:\infSub{X}\to A$ is \emph{perfect} if $h:(\infSub{X},\SHIFT)\to (A,R)$ is a homomorphism, i.e. if $h(N)\mathrel{R} h(\shift{N})$ for every $N\in\infSub{X}$, 
\item A super-sequence $f:F\to A$ is \emph{perfect} if for every $s,t\in F$, $s\tri t$ implies $f(s)\mathrel{R} f(t)$.
\end{enumDefn}
\end{definition}

In particular letting $f=\SHIFT$ in \cref{prop: DichtomyMultiseq}, we obtain the following well-known equivalence.
\begin{corollary}\label{lem:subarr}
For a quasi-order $Q$ the following are equivalent.
\begin{enumThm}
\item $Q$ is \bqo{}, \label{lem:subarr1}
\item every locally constant multi-sequence in $Q$  admits a sub-multi-sequence which is perfect. \label{lem:subarr2}
\item every super-sequence in $Q$ admits a perfect sub-super-sequence. \label{lem:subarr3}
\end{enumThm}
\end{corollary}
\begin{proof}
Let us show that \cref{lem:subarr1} implies \cref{lem:subarr3}. Suppose that $f:F\to Q$ is a super-sequence in a \bqo{} $Q$ where $F$ is a front on $X$. Let $f^{\upcl}:\infSub{X}\to Q$ be the corresponding multi-sequence as defined in \cref{sec:MultisSeqSimpson}. By applying \cref{prop: DichtomyMultiseq} when $f=\SHIFT$ we find $Z\in\infSub{X}$ such that the restriction of $f^{\upcl}$ to $\infSub{Z}$ is perfect. It follows that the restriction of $f$ to $F|Z$ is perfect too.
\end{proof}

\Cref{prop: DichtomyMultiseq} also suggests the following generalisation of the notion of \bqo{} to arbitrary relations:

\begin{definition}
A binary relation $R$ on a discrete space $A$ is a \emph{better-relation} on $A$ if there is no continuous homomorphism $\varphi:(\infSub{\omega},\SHIFT)\to (A,R^{\complement})$.
\end{definition}

This definition first appeared in a paper by \textcite{shelah1982better} and plays an important role in a work by \textcite{marcone1994foundations}. Notice that a better-relation is necessarily reflexive and that a better-quasi-order is simply a transitive better-relation. 

\begin{remark}
One could also consider non discrete analogues of the notion of better-quasi-orders and better-relations. \Textcite{louveau1990quasi} define a \emph{topological better-quasi-order} as a pair $(A,\leq)$, where $A$ is a topological space and $\leq$ is a quasi-order on $A$, such that there is no Borel homomorphism $\varphi:(\infSub{\omega},\SHIFT)\to (A,\leq^{\complement})$. We believe that topological analogs of \bqo{} and better-relations deserve further investigations.
\end{remark}

\subsection{Generalised shifts}
The topological digraph $(\infSub{\omega},\SHIFT)$ is central to the definition of \bqo{}. Indeed a qo $Q$ is \bqo{} if and only if there is no continuous morphism $h:(\infSub{\omega},\SHIFT)\to (Q,\leq^{\complement})$. In general, one can ask for the following:

\begin{problem}\label{prob BqoDef}
Characterise the topological digraphs which can be substituted for $(\infSub{\omega},\SHIFT)$ in the definition of \bqo{}.
\end{problem}

Let us write $(A,R)\contMor (B,S)$ if there exists a continuous homomorphism from $(A,R)$ to $(B,S)$ and $(A,R)\contMorEq(B,S)$ if both $(A,R)\contMor (B,S)$ and $(B,S)\contMor (A,R)$ hold.

Notice that a binary relation $S$ on a discrete space $B$ is a better-relation if and only if $(\infSub{\omega},\SHIFT)\NcontMor (B,S^{\complement})$.
Therefore any topological digraph $(A,R)$ with $(A,R)\contMorEq (\infSub{\omega},\SHIFT)$ can be used in the definition of better-relation in place of $(\infSub{\omega},\SHIFT)$. We do not know whether the converse holds, namely if $(A,R)$ is a topological digraph which can be substituted to $(\infSub{\omega},\SHIFT)$ in the definition of \bqo{}, does it follow that $(A,R)\contMorEq (\infSub{\omega},\SHIFT)$?

We now show that at least the shift map $\SHIFT$ can be replaced by certain \enquote{generalised shifts}. To this end, we first observe that the topological space $\infSub{\omega}$ admits a natural structure of monoid. Following \textcite{solecki2013abstract,promel1986hereditary}, we use the language of increasing injections rather than that of sets.
We denote by $\IIf$ the monoid of embeddings of $(\omega,<)$ into itself under composition,
\[
\IIf=\{f:\omega\to\omega\mid \text{$f$ is injective and increasing}\}.
\]
For every $X\in\infSub{\omega}$, we let $f_{X}\in\IIf$ denote the unique increasing and injective enumeration of $X$. Conversely we associate to each $f\in \IIf$ the infinite subset of $\omega$ given by the range $\{f(n)\mid n\in\omega\}$ of $f$. Therefore the set of substructures of $(\omega,<)$ which are isomorphic to the whole structure $(\omega,<)$, namely $\infSub{\omega}$, is in one-to-one correspondence with the monoid of embeddings of $(\omega,<)$ into itself. Moreover observe that for all $X,Y\in\infSub{\omega}$ we have
\[
X\subseteq Y \SSI \exists g\in \IIf \quad f_{X}=f_{Y}\circ g,
\]
so the inclusion relation on $\infSub{\omega}$ is naturally expressed in terms of the monoid operation. Also, the set $\infSub{X}$ corresponds naturally to the following right ideal:
 \[f_{X}\circ \IIf=\{f_{X}\circ g\mid g\in \IIf\}.\]

As for $\infSub{\omega}$, $\IIf$ is equipped with the topology induced by the Baire space $\bai$ of all functions from $\omega$ to $\omega$. In particular, the composition $\circ:\IIf\times \IIf\to \IIf$, $(f,g)\mapsto f\circ g$ is continuous for this topology. 

Observe now that, in the terminology of increasing injections, the shift map $\SHIFT:\IIf\to \IIf$ is simply the composition on the right with the \emph{successor function} $\Su\in\IIf$, $\Su(n)=n+1$. Indeed for every $X$
\[
f_{\shift{X}}=f_{X}\circ \Su.
\]

This suggests to consider arbitrary injective increasing function $g$, $g\neq \id_{\omega}$, in place of the successor function. For any $g\in \IIf$, we write $\rtrans{g}:\IIf \to \IIf$, $f\mapsto f\circ g$ for the composition on the right by $g$. In particular, $\rtrans{\Su}=\SHIFT$ is the usual shift and in our new terminology we have $(\infSub{\omega},\SHIFT)=(\IIf, \rtrans{\Su})$.

The main result of this section is that these generalised shifts $\rtrans{g}$ are all equivalent as far as the theory of better-relations is concerned. 

\begin{theorem}\label{MainProp}
For every increasing injective function $g\in \IIf$, with $g\neq \id_{\omega}$, we have $(\IIf, \rtrans{g})\contMorEq (\infSub{\omega},\SHIFT)$.
\end{theorem}

\Cref{MainProp} follows from \cref{RmapLemma,EmapLemma} below, but let us first state explicitly some of the direct consequences.

\begin{remark}
Every topological digraph $(A,R)$ has an associated topological graph $(A,R^{\text{s}})$ whose symmetric and irreflexive relation $R^{\text{s}}$ is given by
\[
a\mathrel{R^{\text{s}}}b\SSI a\neq b\text{ and } (a\mathrel{R}b \ \text{or}\ b\mathrel{R}a ).
\]
The Borel chromatic number of topological graphs was first studied by \textcite{Kechris19991}. 
Notably the associated graph of $(\infSub{\omega},\SHIFT)$ has chromatic number $2$ and Borel chromatic number $\aleph_{0}$ (see also the paper by \textcite{di2006canonical}). It directly follows from \cref{MainProp} that for every $g\in \IIf$, with $g\neq\id_{\omega}$, the associated graph of $(\IIf, \rtrans{g})$ also has chromatic number $2$ and Borel chromatic number $\aleph_{0}$.
\end{remark}

\begin{definition}
Let $g\in\IIf$, $R$ a binary relation on a discrete space $A$. We say $(A,R)$ is a \emph{$g$-better-relation} if one of the following equivalent conditions hold:
\begin{enumDefn}
\item for every continuous $\varphi:\IIf\to A$ there exists $f\in \IIf$ such that the restriction $\varphi_{f}:(f\circ \IIf, \rtrans{g})\to (A,R)$ is a continuous morphism,
\item there is no continuous morphism $\varphi:(\IIf, \rtrans{g})\to (A,R^{\complement})$.
\end{enumDefn}
In case $\leq$ is a quasi-order on a discrete space $Q$, we say that $Q$ is \emph{$g$-\bqo{}} instead of $(Q,\leq)$ is a $g$-better-relation.
\end{definition}

 Of course this notion trivialises for $g=\id_{\omega}$, since an $\id_{\omega}$-better-relation is simply a reflexive relation. Moreover better relation corresponds to $\Su$-better-relation.

\begin{theorem}\label{g-BQO}
Let $g\in\IIf\setminus\{\id_{\omega}\}$, $R$ a binary relation on a discrete space $A$. Then $R$ is a $g$-better-relation if and only if $R$ is a better-relation. In particular, a quasi-order $(Q,\leq)$ is $g$-\bqo{} if and only if $(Q,\leq)$ is \bqo{}. 
\end{theorem}

\begin{corollary}
A qo $Q$ is \bqo{} if and only if for every locally constant $\varphi:\IIf\to Q$ and every $g\in \IIf$ there exists $f\in \IIf$ such that
\[
\varphi(f)\leq \varphi(f\circ g).
\]
\end{corollary}

As a corollary we have the following strengthening of \cref{lem:subarr} which is obtained by repeated applications of \cref{prop: DichtomyMultiseq}.

\begin{proposition}\label{prop GenPerfect}
Let $Q$ be \bqo{} and $\varphi:\IIf\to Q$ be locally constant. For every finite subset $\mathcal{G}$ of $\IIf$ there exists $h\in \IIf$ such that the restriction $\varphi:h\circ \IIf \to Q$ is perfect with respect to every member of $\mathcal{G}$, i.e. for every $f\in \IIf$ and every $g\in \mathcal{G}$
\[
\varphi(h\circ f)\leq \varphi(h\circ f\circ g).
\]
\end{proposition}

Getting a result of this kind was one of our motivations for proving \cref{MainProp}. 

Finally here are the two lemmas which yield the proof of \cref{MainProp}.

\begin{lemma}\label{RmapLemma}
Let $g\in\IIf\setminus\{ \id_{\omega}\}$. Then $(\IIf,\rtrans{g})\contMor (\IIf,\rtrans{\Su})$, i.e. there exists a continuous map $\rho:\IIf \to \IIf$ such that for every $f\in \IIf$
\[
\rho(f\circ g)=\rho(f)\circ \Su.
\]
\end{lemma}

\begin{proof}
Since $g\neq \id_{\omega}$, there exists $k_{g}=\min\{k\in\omega\mid k<g(k)\}$. 
Define $G:\omega\to \omega$ by $G(n)=g^{n}(k_{g})$, where $g^{0}=\id_{\omega}$ and $g^{n+1}=g\circ g^{n}$. Clearly $G\in \IIf$. We let $\rho(f)=f\circ G$ for every $f\in \IIf$. The map $\rho:\IIf\to \IIf$ is continuous and for every $f\in \IIf$ and every $n$ we have \[
\rho(f\circ g)(n)=f\circ g\circ g^{n}(k_{g})=f\circ G(n+1)=(\rho(f)\circ \Su)(n).\qedhere
\] 
\end{proof}

\begin{lemma}\label{EmapLemma}
Let $g\in\IIf\setminus\{\id_{\omega}\}$. Then $(\IIf,\rtrans{\Su})\contMor(\IIf,\rtrans{g})$, i.e. there exists a continuous map $\sigma : \IIf \to \IIf$ such that for every $f\in \IIf$
\[
\sigma (f\circ \Su)=\sigma (f)\circ g.
\]
\end{lemma}
\begin{proof}
Let $k_{g}=\min\{k\mid k<g(k)\}$. As in the proof of the previous Lemma we define $G\in \IIf$ by $G(n)=g^{n}(k_{g})$. For every $f\in\IIf$ and every $l\in\omega$, we let
\[
\sigma (f)(l)=\begin{cases}
l &\text{if $l<G(0)$,}\\
g^{f(n)-n}(l) &\text{if $G(n)\leq l<G(n+1)$, for $n\in\omega$.}
\end{cases}
\]

Let us check that $\sigma (f)$ is indeed an increasing injection from $\omega$ to $\omega$ for every $f\in \IIf$. Since $\sigma (f)$ is increasing and injective on each piece of its definition, it is enough to make the two following observations. Firstly, if $l<G(0)$, then
\[
\sigma (f)(l)=l<G(0)\leq G\circ f(0)=g^{f(0)}(G(0))=\sigma (f)(G(0)).
\]
Secondly, if $G(n)\leq l<G(n+1)$ then 
\begin{multline*}
\sigma (f)(l)=g^{f(n)-n}(l)<g^{f(n)-n}(G(n+1))\\=g^{f(n)+1}(k_{g})\leq g^{f(n+1)}(k_{g})=G ( f(n+1)),
\end{multline*}
but we have
\[
G(f(n+1))=g^{f(n+1)-(n+1)}(G(n+1))=\sigma (f)(G(n+1)).
\]
One can easily check that $\sigma : \IIf \to \IIf$ is continuous.
Now on the one hand
\[
\sigma (f\circ \Su)(l)=\begin{cases}
l &\text{if $l<G(0)$,}\\
g^{f(n+1)-n}(l) &\text{if $G(n)\leq l<G(n+1)$, for $n\in\omega$.}
\end{cases}
\]
and on the other hand
\[
\sigma (f)(g(l))=\begin{cases}
g(l) &\text{if $g(l)<G(0)$,}\\
g^{f(n)-n}(g(l)) &\text{if $G(n)\leq g(l)<G(n+1)$.}
\end{cases}
\]
By definition of $G$, we have $g(l)<G(0)$ if and only if $l=g(l)$. 
Moreover if $G(n)\leq l<G(n+1)$ then we have $G(n+1)\leq g(l)<G(n+2)$ and so
\[
\sigma ( f\circ \Su)(l)=g^{f(n+1)-n}(l)=g^{f(n+1)-(n+1)}(g(l))=\sigma (f)(g(l)),
\]
which proves the Lemma. 
\end{proof}




\printbibliography

\end{document}